\documentclass[epj]{my-svjour}

\usepackage{amsmath,amssymb,amsfonts}
\usepackage[labelfont=bf]{caption}
\usepackage{url}
\usepackage{graphicx}
\usepackage{epsfig}
\usepackage{subcaption}
\usepackage{booktabs}

\allowdisplaybreaks[0]

\DeclareUnicodeCharacter{2212}{-} 

\newtheorem{thm}{\bf Theorem}

\begin{document}

\title{Stability analysis and optimal control of a fractional HIV-AIDS epidemic model 
with memory and general incidence rate}

\titlerunning{Stability analysis and optimal control of a fractional HIV-AIDS epidemic model}

\author{Adnane Boukhouima\inst{1}\\
\email{adnaneboukhouima@gmail.com}\\[0.2cm]
El Mehdi Lotfi\inst{1}\\
\email{lotfiimehdi@gmail.com}\\[0.2cm]
Marouane Mahrouf\inst{1}\\
\email{marouane.mahrouf@gmail.com}\\[0.2cm]
Silv\'{e}rio Rosa\inst{2}\\
\email{rosa@ubi.pt}\\[0.2cm]
Delfim F. M. Torres\inst{3}\\
\email{delfim@ua.pt}\\[0.2cm]
Noura Yousfi\inst{1}\\
\email{nourayousfi@hotmail.com}}

\authorrunning{A. Boukhouima et al.}

\mail{Delfim F. M. Torres, \email{delfim@ua.pt}}

\institute{Laboratory of Analysis, Modeling and Simulation (LAMS),
Faculty of Sciences Ben M'sik, Hassan II University,
P.O Box 7955 Sidi Othman, Casablanca, Morocco
\and
Instituto de Telecomunica\c{c}\~{o}es and Department of Mathematics,
Universidade da Beira Interior, 6201-001 Covilh\~{a}, Portugal
\and
Center for Research and Development in Mathematics and Applications (CIDMA),
Department of Mathematics, University of Aveiro, 3810-193 Aveiro, Portugal}

\date{Submitted: 02-Feb-2020 / Revised: 20-Oct-2020 and 24-Nov-2020 / Accepted: 08-Dec-2020}


\abstract{We investigate the celebrated mathematical
SICA model but using fractional differential equations
in order to better describe the dynamics of HIV-AIDS infection. 
The infection process is modelled by a general functional response 
and the memory effect is described by the Caputo fractional derivative. 
Stability and instability of equilibrium points are determined 
in terms of the basic reproduction number. 
Furthermore, a fractional optimal control
system is formulated and the best strategy for minimizing
the spread of the disease into the population
is determined through numerical simulations based
on the derived necessary optimality conditions.
\keywords{HIV-AIDS infection; Fractional calculus;
Fractional differential equations; Stability.}
\PACS{{02.30.Hq}{Ordinary differential equations}
\and
{02.30.Yy}{Control theory}
\and
{02.60.−x}{Numerical approximation and analysis}
\and
{87.00.00}{Biological and medical physics}
\and
{87.19.Xx}{Diseases}
\and
{87.23.Kg}{Dynamics of evolution}.}}

\maketitle


\section{Introduction}

Acquired Immune Deficiency Syndrome (AIDS) is a chronic infectious disease
caused by the human immunodeficiency virus (HIV). The virus attack and
destruct the immune response system, which plays a crucial role to defend
the human body against viral pathogens. The last statistics of
The Joint United Nations Programme on HIV and AIDS (UNAIDS)
show that 36.9 million people were living with HIV, where 21.7
million individuals were accessing anti-retro-viral therapy (ART)
and 1.8 million became newly infected with HIV \cite{UNAIDS}.
Therefore, the world is now facing huge challenges and should
be committed to provide appropriate preventive strategies
in order to control the AIDS epidemic process.

In the course of history, mathematical modelling of natural phenomena,
described by ordinary differential equations (ODE), has proven valuable
in analysing various diseases dynamics, such as HIV/AIDS, Malaria,
and Tuberculosis, and also plays an important role in better understanding
the global behaviour of epidemiological models. However, memory has a crucial
role in the evolution and control of any epidemic process. The experience
or knowledge of people about the spread of a disease in the past,
affects their response. If people know about the history of a certain disease
in their environment, then they may use different precautions, such as vaccination
and treatment \cite{Saeedian}. Consequently, incorporating memory seems
very appropriate to study such epidemic models. This is done here through
the use of fractional differentiation.

Fractional derivatives, which provide a generalization of the integer order derivative
to an arbitrary order, has become an adequate mathematical tool to characterize
the memory effect of complex systems. This particular property is neglected
by integer order derivatives, which explains why fractional calculus
is nowadays widely applied to model various dynamical processes in different fields
of science and engineering, such as mechanics, image processing, viscoelasticity,
bioengineering, finance, psychology, and biology
\cite{Rossikhin,Marks,Jia,Magin,Scalas1,Song1,Cole,Agarwal,AGARWAL2020122769}.
The advantage of using systems of fractional differential equations (FDE) over ODE systems
is that they provide an excellent tool for the description of memory and hereditary properties.
Indeed, the fractional derivative is a non-local operator, in contrast with the integer derivative,
which means that if we want to compute a fractional derivative at some point $t = t_{1}$
we need all its history from the starting point $t = t_{0}$ up to the point $t = t_{1}$.
Furthermore, another feature of FDE is that their stability region is larger than ODE,
which can help to reduce errors arising from the neglected parameters
in modelling real life problems \cite{Rihan}.

Since modelling of dynamic systems by ODE cannot precisely describe experimental
and measurement data \cite{Sun}, FDE are now being extensively applied to study
and control the dynamics of infectious diseases. In \cite{Rihan1}, Rihan et al.
analyse a fractional model for HCV dynamics in presence of interferon-$ \alpha$ (IFN)
treatment. According to numerical simulations and the real data, they also confirm
that the FDE are better descriptors of HCV systems than ODE. In the study
of Arafa et al. \cite{Arafa}, the authors compared between the results
of the fractional order model, the results of the integer model,
and the measured real data obtained from 10 patients during primary HIV infection,
and they proved that the results of the fractional order model give better
predictions to the plasma virus load of the patients than those of the integer order model.
In the work of Wojtak et al. \cite{Wojtak}, the authors investigate the
uniform asymptotic stability of the unique endemic equilibrium for a Caputo
fractional-order tuberculosis (TB) model. They confirm that the proposed
fractional-order model provides richer and more flexible results when
compared with the corresponding integer-order TB model. In \cite{Parra},
the authors propose a non-linear fractional-order model to explain and
understand the outbreaks of influenza A (H1N1) worldwide. They show
that the fractional-order model gives wider peaks and leads to better
approximations for the real epidemic data. The authors in \cite{Pooseh}
propose a fractional-order model and show, through numerical simulations,
that the fractional models fit better the first dengue epidemic recorded
in the Cape Verde islands off the coast of West Africa when compared
with ODE models. For fractional optimal control problems (FOCP),
we can cite the work of Sweilam et al., where a fractional optimal control
model for tuberculosis infection, including the impact of diabetes
and resistant strains, is studied \cite{Sweilam}. Also, we mention the study
of Rosa and Torres, where optimal control of a fractional order
epidemic model with application to human respiratory syncytial virus infection
is proposed and studied \cite{Rosa}. In the case of HIV-AIDS infection,
Kheiri and Jafari propose and analyse a fractional optimal control
of an HIV/AIDS epidemic model with random testing and contact tracing \cite{Kheiri1}.
On the other hand, a fractional malaria transmission model is investigated
by Pinto and Tenreiro Machado \cite{Pinto} on the basis of optimal control techniques. 
Other works can be found in \cite{Boukhouima,Mouaouine,Boukhouima2019,AGARWAL2020124243}.

From all this biological and mathematical considerations, 
here we are interested to investigate the transmission process of HIV-AIDS infection, 
taking into account the memory effect that exists in most dynamical systems. 
The infection process is modelled by a general incidence rate which covers, 
under some hypothesis, the most functional response existing in the literature. 
Using Lyapunov functionals and the fractional invariance principal, 
we prove that the global dynamics of the model is determined by the basic reproduction number.
Furthermore, in order to minimize the spread of the disease into the population, 
a fractional optimal control is formulated and numerically solved based on Moroccan data.

We organized the paper as follows. 
In Section~\ref{sec:2}, some properties of the solutions 
are given and existence conditions of the equilibrium points
are discussed. The stability analysis of the equilibria is studied
in Section~\ref{sec:3}, while in Section~\ref{sec:optctrl} the fractional
optimal control of the model is investigated. An application of our
analysis to Morocco data is given in Section~\ref{sec:5}. We end with
Section~\ref{sec:conc} of conclusions.


\section{Well-possessedness of the model}
\label{sec:2}

In this section, we propose a fractional SICA epidemic model 
with general incidence rate (Section~\ref{subsec:PR}) 
and give some preliminary but fundamental results
that include the well-possessedness of the model
and existence conditions of the possible equilibria
(Section~\ref{subsec:PR}).


\subsection{Mathematical model}
\label{subsec:MM}

Taking into consideration the memory effect presented by Caputo fractional derivatives, 
we propose the following SICA epidemic model with general incidence rate:
\begin{equation}
\label{1}
\begin{cases}
_{0}^{C}D_{t}^{\alpha}S(t) = \Lambda  - \mu S(t)- f\left(S(t),I(t)\right)I(t),\\[0.2 cm]
_{0}^{C}D_{t}^{\alpha}I(t) = f\left(S(t),I(t)\right)I(t) - (\rho + \phi + \mu)I(t)
+ \sigma A(t)  + \omega C(t), \\[0.2 cm]
_{0}^{C}D_{t}^{\alpha}C(t) = \phi I(t) - (\omega + \mu)C(t),\\[0.2 cm]
_{0}^{C}D_{t}^{\alpha}A(t) =  \rho \, I(t) - (\sigma + \mu + d) A(t),
\end{cases}
\end{equation}
where $_{0}^{C}D_{t}^{\alpha}$ represents the Caputo fractional derivative
of order $ 0<\alpha\leq1$ defined for an arbitrary function $ \varphi $ \cite{Podlubny} by
$$
_{0}^{C}D_{t}^{\alpha}\varphi(t)= \dfrac{1}{\Gamma(1-\alpha)}\int^{t}_{0}\dfrac{\varphi'(x)}{(t-x)^{\alpha}}dx.
$$
Note that when $\alpha\rightarrow1 $ system (\ref{1}) becomes a classical system of ODEs.

\begin{remark}
The Caputo derivative is a good choice in order to include long-term memory effects. 
Indeed, the power-law function $ (t-x)^{-\alpha} $, that appears in its definition,
exhibits a slow decay and the state of the system at quite early times also contribute 
to the evolution of the system. This type of kernel guarantees the existence of scaling 
features as it is often intrinsic in natural phenomena. Hence, fractional derivatives, 
when introducing a convolution integral with a power-law memory kernel, are useful to describe 
memory effects in dynamical systems. The decaying rate of the memory kernel 
(a time-correlation function) depends on $ \alpha$. A lower value of $ \alpha $ 
corresponds to more slowly-decaying time-correlation functions (long memory). In some sense, 
the strength of the memory is controlled by $ \alpha $. As $ \alpha\rightarrow 1 $, the influence 
of memory decreases: the system tends toward a memoryless system. While modelling various memory phenomena, 
one observes that memory processes usually consist of two stages. One is short with permanent retention,
while the other is governed by a simple model of fractional derivative. It has been shown that fractional models
perfectly fits the test data of memory phenomena in different disciplines, 
for example in mechanics, but also in biology and psychology. 
The interested reader in these issues is referred to \cite{Du2013}.
\end{remark}

The variables $S$, $I$, $C$ and $A$ represent individuals, respectively, susceptible,
HIV infected with no clinical symptoms of AIDS, HIV infected under ART treatment,
with a viral load remaining low, and HIV infected with AIDS clinical symptoms.
The susceptible population is increased by the recruitment of individuals
at a rate $\Lambda$, while $\mu$ is the natural death rate of all individuals.
Susceptible individuals acquire HIV infection at a rate $f(S,I)$ by following
effective contact with those in the class $I$. HIV-infected individuals
with no AIDS symptoms $I$ progress to the class $C$ at a rate $\phi$
and, if they do not follow treatment, to the class $A$ at a rate $\rho$.
HIV-infected individuals with AIDS symptoms are treated for HIV
at rate $\sigma$. Individuals in the class $C$ that do not maintain
treatment, leave to the class $I$ at a rate $\omega$. We assume that
only HIV-infected individuals with AIDS symptoms $A$ suffer from
an AIDS induced death rate, denoted by $d$. As in \cite{Hattaf},
the general incidence function $f(S,I)$ is assumed to be continuously
differentiable in the interior of $\mathbb{R}^{2}_{+}$ and to satisfy
the following hypotheses:
\begin{gather}
\label{H1} \tag{$H_{1}$} 
f(0,I)=0,
\hspace*{0.2 cm} \text{ for all } I \geq 0,\\
\label{H2}\tag{$H_{2}$} \frac{ \partial f}{\partial S}(S,I)> 0,
\hspace*{0.2 cm} \text{ for all } S>0\ \text{ and } \ I \geq 0,\\
\label{H3}\tag{$H_{3}$} \frac{ \partial f}{\partial I}(S,I)
\leq 0 \hspace*{0.2 cm},\hspace*{0.2 cm} \text{ for all }\ S\geq 0\
\text{ and } \ I \geq 0.
\end{gather}
Biologically, the three hypotheses $H_{1}$, $H_{2}$ and $H_{3}$ are reasonable. 
Indeed, the first means that the incidence function is equal to zero 
if there are no susceptible individuals. The second one signifies that 
the incidence rate is increasing when the number of infected 
individuals are constant and the number of susceptible individuals increases. 
This means that the higher the number of susceptible individuals, 
the higher the average number of individuals infected over time. 
The last hypothesis means that the higher the number of infected individuals, 
the lower the average number of infected individuals over time.


\subsection{Preliminary results}
\label{subsec:PR}

Since model (\ref{1}) describes the evolution of population,
we need to prove that the solutions are non-negative and bounded
for all time. These properties imply the global existence of solutions.
For biological considerations, we assume that the initial conditions satisfy
\begin{equation}
\label{IC}
S(0)=S_0\geq0, \
I(0)=I_0\geq0, \
C(0)=C_0\geq0, \
A(0)=A_0\geq 0.
\end{equation}

\begin{thm}
For any initial conditions satisfying (\ref{IC}), system (\ref{1})
has a unique  solution on $\left[0,+\infty\right)$. Moreover, this
solution remains non-negative and bounded for all $ t\geq0$.
In addition, we have
$$
N(t)\leq N(0)+\dfrac{\Lambda}{\mu},
$$
where $N(t)=S(t)+I(t)+C(t)+A(t)$.
\end{thm}

\begin{proof}
First, system (\ref{1}) can be written as follows:
\begin{equation}
\label{3}
_{0}^{C}D_{t}^{\alpha}X(t)= F(X) ,
\end{equation}
where
\begin{center}
$X(t)=\begin{pmatrix}
S(t) \\
I(t) \\
C(t)\\
A(t)
\end{pmatrix}$
and $F(X)= \begin{pmatrix}
\Lambda  - \mu S(t)- f\left(S(t),I(t)\right)I(t)\\
f\left(S(t),I(t)\right)I(t) - (\rho + \phi + \mu)I(t)
+ \sigma A(t)  + \omega C(t) \\
\phi I(t) - (\omega + \mu)C(t)\\
\rho \, I(t) - (\sigma + \mu + d) A(t)
\end{pmatrix}.$
\end{center}
Clearly, function $F$ satisfies the conditions given in \cite{Lin}.
Then, there exists a unique local solution
of the initial value problem \eqref{3}.
Now, we show that the non-negative orthant
$\mathbb{R}^{4}_{+}=\lbrace X\in \mathbb{R}^{4} : X \geq 0 \rbrace$ is a
positively invariant set. We denote by $t^{*}$ the first time 
at which at least one of the variables is equal to zero:
$$
t^{*}=\min\lbrace t>0:S(t)I(t)C(t)A(t)=0\rbrace.
$$
We discuss four cases. (i) If $ S(t^{*})=0$, then it follows that
$I(t)\geq 0, C(t)\geq 0$ and $A(t)\geq 0$ when $t \in [0,t^{*}]$.
From the first equation of system (\ref{1}), we have
$$
_{0}^{C}D_{t}^{\alpha}S(t)|_{t=t^{*}} = \Lambda>0.
$$
By the generalized mean value theorem \cite{Odibat},
$S(t)$ is a non-decreasing function for $t\in (t^{*}-\epsilon,t^{*}]$,
where $\epsilon$ is sufficiently small. So, $S(t)<0$ for
$t\in (t^{*}-\epsilon,t^{*}]$, which is a contradiction with $ S(t)>0 $
when $t\in (0,t^{*})$.
(ii) Let $ I(t^{*})=0$. In this case, we have $ S(t)\geq 0, C(t)\geq 0 $
and $ A(t)\geq 0 $ when $ t \in [0,t^{*}] $. From the second equation
of system (\ref{1}), we have
$$
_{0}^{C}D_{t}^{\alpha}I(t)|_{t=t^{*}} = \sigma A(t)  + \omega C(t)\geq0.
$$
Hence, function $ I(t) $ is non-decreasing for $ t\in (t^{*}-\epsilon,t^{*}] $,
where $  \epsilon $ is sufficiently small. Thus, $ I(t)\leq0 $ for
$ t\in (t^{*}-\epsilon,t^{*}] $. This is in contradiction
with $ I(t)>0 $ when $ t\in (0,t^{*})$.
(iii) Let $ C(t^{*})=0$. Then, $ S(t)\geq 0, I(t)\geq 0 $ and
$ A(t)\geq 0 $ when $ t \in [0,t^{*}] $. From the third equation
of system \eqref{1}, we have
$$
_{0}^{C}D_{t}^{\alpha}C(t)|_{t=t^{*}} = \phi I(t) \geq0.
$$
Therefore, function $ C(t) $ is non-decreasing for $ t\in (t^{*}-\epsilon,t^{*}] $,
where $  \epsilon $ is sufficiently small. So $ C(t)\leq0 $ for
$ t\in (t^{*}-\epsilon,t^{*}] $. That is a contradiction
with $ C(t)>0 $ when $ t\in (0,t^{*}) $.
(iv) Let $ A(t^{*})=0$. Hence, $ S(t)\geq 0, I(t)\geq 0 $ and
$ C(t)\geq 0 $ when $ t \in [0,t^{*}] $. From the last equation
of system \eqref{1}, we have
$$
_{0}^{C}D_{t}^{\alpha}A(t)|_{t=t^{*}} = \rho I(t) \geq0.
$$
As a result, function $ A(t) $ is non-decreasing for $ t\in (t^{*}-\epsilon,t^{*}] $,
where $  \epsilon $ is sufficiently small. So $ A(t)\leq0 $ for $ t\in (t^{*}-\epsilon,t^{*}] $,
which is in contradiction with $ A(t)>0 $ when $ t\in (0,t^{*}) $.

Next, we prove the boundedness of solutions. By adding together
all the equations of system \eqref{1}, one has that
\begin{align*}
_{0}^{C}D_{t}^{\alpha}N(t)\leq \Lambda - \mu N(t).
\end{align*}
Hence,
\begin{align*}
N(t)\leq N(0) E_{\alpha}(-\mu t^{\sigma})+\dfrac{\Lambda}{\mu}[1-E_{\alpha}(-\mu t^{\alpha})].
\end{align*}
Since $0\leq E_{\alpha}(-\mu t^{\alpha})\leq 1 $, we obtain
$$
N(t)\leq N(0)+\dfrac{\Lambda}{\mu}.
$$
Consequently, the solutions of system \eqref{1} are bounded for $ t\geq0 $.
Finally, the existence and uniqueness of solution for the initial
value problem \eqref{3} in $ [0,+\infty) $ is deduced from
\cite[Theorem 3.1 and Remark 3.2]{Lin}.
\hfil \qed
\end{proof}

Now, we investigate the existence of equilibria of \eqref{1}.
It is easy to see that system \eqref{1}
has a disease-free equilibrium of the form
\begin{equation}
\label{eq:Ef}
E_{f}=\left(\dfrac{\Lambda}{\mu},0,0,0 \right).
\end{equation}
Therefore, the basic reproduction number $ R_{0} $
of system \eqref{1} is given by
$$
R_{0}=\dfrac{f\left( \dfrac{\Lambda}{\mu},0\right) \xi_{2}\xi_{3}}{\mathcal{D}},
$$
where
\begin{align*}
\xi_{2}&=\omega+\mu,\\
\xi_{3}&= \sigma + \mu +d,\\
\mathcal{D}&=\mu[\xi_{2}(\xi_{3}+\rho)+\phi\xi_{3}+\rho d]+\rho\omega d.
\end{align*}
Biologically, this number represents the average of new infected individuals
produced by a single HIV-infected/AIDS individual
on contact in a completely susceptible population.

The other equilibria satisfy the following system:
\begin{align}
\label{4a}
\begin{cases}
\Lambda  - \mu S(t)- f\left(S(t),I(t)\right)I(t)=0,\\
f\left(S(t),I(t)\right)I(t) - \xi_{1}I(t)
+ \sigma A(t)  + \omega C(t)=0,\\
\phi I(t) - \xi_{2}C(t)=0,\\
\rho \, I(t) - \xi_{3} A(t)=0,
\end{cases}
\end{align}
where $\xi_{1}=\rho + \phi + \mu$.

Since FDEs have the same equilibrium points as ODEs counterparts,
then we have the following direct result from \cite{lotfi}.

\begin{thm}
\begin{description}
\item[(i)]If $ R_{0}\leq 1 $, then system \eqref{1} has
a unique disease-free equilibrium of form \eqref{eq:Ef}.
\item[(ii)] If $ R_{0}>1 $, then the disease-free equilibrium
is still present and system \eqref{1} has a unique endemic equilibrium
of form $ E^{*} = (S^{*},I^{*},C^{*},A^{*}) $ with
$ S^{*}\in  \left(0,\dfrac{\Lambda}{\mu} \right) $, $ I^{*}>0 $,
$ C^{*}>0 $, and $A^{*}>0$.
\end{description}
\end{thm}


\section{Global stability}
\label{sec:3}

The aim of this section is to establish the global stability of equilibria
of \eqref{1} by using the fractional La-Salle's invariance principle and
an important lemma presented in \cite{De-Leon,Huo}. Firstly, we have
the following global stability result for the infection-free equilibrium $E_f$.

\begin{thm}
The disease-free equilibrium $E_f$ is globally
asymptotically stable if $R_0\leq1$.
\end{thm}

\begin{proof}
For the global stability of $E_f$, we construct the following Lyapunov functional:
\begin{equation*}
V_{1}(S,I,C,A)=S-S_{0}-\int_{S_{0}}^{S}\frac{f(S_{0},0)}{f(X,0)}dX
+I+\dfrac{\omega}{\xi_2}C+\dfrac{\sigma}{\xi_3}A,
\end{equation*}
where $S_0=\dfrac{\Lambda}{\mu}$. Obviously, functional $V_{1}$
is non-negative. Computing the fractional time derivative
of $V_{1}$ along the solution of \eqref{1}, we get
\begin{equation*}
_{0}^{C}D_{t}^{\alpha}V_{1}
=_{0}^{C}D_{t}^{\alpha}\left( S-S_{0}
-\int_{S_{0}}^{S}\frac{f(S_{0},0)}{f(X,0)}dX\right)
+ \:_{0}^{C}D_{t}^{\alpha}I+\dfrac{\omega}{\xi_2}\:_{0}^{C}D_{t}^{\alpha}C
+\dfrac{\sigma}{\xi_3}\:_{0}^{C}D_{t}^{\alpha}A.
\end{equation*}
We start by proving that
\begin{align}
\label{5}
_{0}^{C}D_{t}^{\alpha}\left( S-S_{0}
-\int_{S_{0}}^{S}\frac{f(S_{0},0)}{f(X,0)}dX\right)
\leq \left(1-\dfrac{f(S_{0},0)}{f(S,0)}\right)\: _{0}^{C}D_{t}^{\alpha}S.
\end{align}
Inequality \eqref{5} can be reformulated as follows:
 \begin{equation}
 \label{10}
_{t_{0}}^{C}D_{t}^{\alpha}S(t)-f(S,0)_{t_{0}}^{C}D_{t}^{\alpha}\left[
\int_{S_{0}}^{S}\frac{1}{f(X,0)}dX\right]\leq 0.
\end{equation}
Using the definition of the Caputo fractional derivative, we have
\begin{equation*}
_{t_{0}}^{C}D_{t}^{\alpha}S(t)
=\dfrac{1}{\Gamma(1-\alpha)}
\int^{t}_{t_{0}}\dfrac{S'(y)}{(t-y)^{\alpha}} dy
\end{equation*}
and
\begin{equation*}
_{t_{0}}^{C}D_{t}^{\alpha}\left[
\int_{S_{0}}^{S}\frac{1}{f(X,0)}dX\right]
=\dfrac{1}{\Gamma(1-\alpha)}\int^{t}_{t_{0}}
\dfrac{S'(y)}{(t-y)^{\alpha}f(S(y),0)} dy.
\end{equation*}
Consequently, the inequality (\ref{10}) can be written as
\begin{equation}
\label{11}
\dfrac{1}{\Gamma(1-\alpha)}\int^{t}_{t_{0}}
\dfrac{S'(y)}{(t-y)^{\alpha}}\left(
1-\dfrac{f(S(t),0)}{f(S(y),0)}\right) dy\leq 0.
\end{equation}
Now, we show that inequality (\ref{11}) holds. Denoting
$$
\Psi(t)=\dfrac{1}{\Gamma(1-\alpha)}
\int^{t}_{t_{0}}\dfrac{S'(y)}{(t-y)^{\alpha}}\left(
1-\dfrac{f(S(t),0)}{f(S(y),0)}\right) dy,
$$
we integrate by parts by defining
\begin{equation*}
v(y)=\dfrac{(t-y)^{-\alpha}}{\Gamma(1-\alpha)},
\quad \:v'(y)=\dfrac{\alpha(t-y)^{-(\alpha+1)}}{\Gamma(1-\alpha)}
\end{equation*}
and
\begin{equation*}
w'(y)=S'(y)\left( 1-\dfrac{f(S(t),0)}{f(S(y),0)}\right),
\quad \:w(y)=S(y)-S(t)-\int^{S(y)}_{S(t)}\dfrac{f(S(t),0)}{f(X,0)}dX,
\end{equation*}
to obtain
\begin{equation}
\label{12}
\begin{split}
\Psi(t)
&= \left[ \dfrac{(t-y)^{-\alpha}}{\Gamma(1-\alpha)}\left(S(y)-S(t)
-\int^{S(y)}_{S(t)}\dfrac{f(S(t),0)}{f(X,0)}dX\right) \right] ^{y=t}\\
&\quad -\dfrac{(t-t_{0})^{-\alpha}}{\Gamma(1-\alpha)}\left(S(t_{0})-S(t)
-\int^{S(t_{0})}_{S(t)}\dfrac{f(S(t),0)}{f(X,0)}dX\right)\\
&\quad -\int^{t}_{t_{0}}\dfrac{\alpha(t-y)^{-(\alpha+1)}}{\Gamma(1-\alpha)}\left(
S(y)-S(t) -\int^{S(y)}_{S(t)}\dfrac{f(S(t),0)}{f(X,0)}dX\right) dy.
\end{split}
\end{equation}
We can easily see that the first term in (\ref{12}) is undefined $ (\frac{0}{0}) $.
We analyse the corresponding limit. By H\^opital's rule, we get
\begin{equation*}
\lim_{y\to t}\dfrac{(t-y)^{-\alpha}}{\Gamma(1-\alpha)}\left(S(y)
-S(t) -\int^{S(y)}_{S(t)}\dfrac{f(S(t),0)}{f(X,0)}dX\right)
=\lim_{y\to t}\dfrac{S'(y)\left( 1-\dfrac{f(S(t),0)}{
f(S(y),0)}\right)}{-\alpha\Gamma(1-\alpha)(t-y)^{\alpha-1}}=0.
\end{equation*}
Hence,
\begin{equation}
\begin{split}
\Psi(t)
&=-\dfrac{(t-t_{0})^{-\alpha}}{\Gamma(1-\alpha)}\left(S(t_{0})-S(t)
-\int^{S(t_{0})}_{S(t)}\dfrac{f(S(t),0)}{f(X,0)}dX\right)\\
&\qquad -\int^{t}_{t_{0}}\dfrac{\alpha(t-y)^{-(\alpha+1)}}{\Gamma(1-\alpha)}\left(
S(y)-S(t) -\int^{S(y)}_{S(t)}\dfrac{f(S(t),0)}{f(X,0)}dX\right) dy.
\end{split}
\end{equation}
Then,
\begin{equation}
\Psi(t)=\dfrac{1}{\Gamma(1-\alpha)}
\int^{t}_{t_{0}}\dfrac{S'(y)}{(t-y)^{\alpha}}\left(
1-\dfrac{f(S(t),0)}{f(S(y),0)}\right) dy\leq 0.
\end{equation}
As a result, the inequality (\ref{11}) is satisfied. Consequently,
\begin{align*}
_{0}^{C}D_{t}^{\alpha}V_{1}(t)
&\leq \left(1-\dfrac{f(S_{0},0)}{f(S,0)}\right)\:_{0}^{C}D_{t}^{\alpha}
S+\:_{0}^{C}D_{t}^{\alpha}I+\dfrac{\omega}{\xi_2}\:
_{0}^{C}D_{t}^{\alpha}C+\dfrac{\sigma}{\xi_3}\:_{0}^{C}D_{t}^{\alpha}A\\
&\leq \mu \left(1-\dfrac{f(S_{0},0)}{f(S,0)}\right)(S_0-S)
+I\left[\dfrac{f(S,I)}{f(S,0)}f(S_0,0)-\left(\xi_1
-\dfrac{\omega \phi}{\xi_2}-\dfrac{\sigma \rho}{\xi_3}\right)\right]\\
&\leq \mu \left(1-\dfrac{f(S_{0},0)}{f(S,0)}\right)(S_0-S)
+\dfrac{\mathcal{D}}{\xi_2 \xi_3}I\left(\dfrac{f(S,I)}{f(S,0)}R_0-1\right)\\
&\leq \mu \left(1-\dfrac{f(S_{0},0)}{f(S,0)}\right)(S_0-S)
+\dfrac{\mathcal{D}}{\xi_2 \xi_3}I\left(R_0-1\right).
\end{align*}
Since $f$ is an increasing function with respect to $S$, we have
\begin{align*}
1-\dfrac{f(S_{0},0)}{f(S,0)}
&\geq 0 \quad \text{for}\quad S\geq S_0,\\
1-\dfrac{f(S_{0},0)}{f(S,0)}
&< 0 \quad \text{for}\quad S<S_0.
\end{align*}
We finally get 
\begin{equation*}
\left(1-\dfrac{f(S_{0},0)}{f(S,0)}\right)(S_0-S)\leq 0.
\end{equation*}
Under the assumption $ R_{0}\leq 1 $, it follows that
$ _{0}^{C}D_{t}^{\alpha}V_{1}\leq0 $. Moreover,
the largest compact invariant set in
$\lbrace (S,I,C,A)\in \mathbb{R}^{4}:\, _{0}^{C}D_{t}^{\alpha}V_{1}\leq0 \rbrace$
is the singleton $ E_{f} $. Accordingly, by LaSalle invariance principle,
the infection-free equilibrium $E_{f}$ is
globally asymptotically stable when $ R_{0}\leq 1 $.
\hfil \qed
\end{proof}

Now, we focus on the stability of the endemic equilibrium $E^*$.
For that, we assume that the function $f$ satisfies the following condition:
\begin{gather}
\label{H4}\tag{$H_{4}$}
\bigg(1-\dfrac{f(S,I)}{f(S,I^*)}\bigg)\bigg(\dfrac{f(S,I^*)}{f(S,I)}
-\dfrac{I}{I^*}\bigg)\leq0, \ \text{ for all } \ S,I>0.
\end{gather}

\begin{thm}
\begin{description}
\item[(i)]If $R_0>1$, then $E_f$ becomes unstable.
\item[(ii)]If $R_0>1$ and \eqref{H4} holds, then
the endemic equilibrium $E^*$ is globally asymptotically stable.
\end{description}
\end{thm}

\begin{proof}
The proof of the instability of $ E_{f} $ is based on the computation of the Jacobian
matrix of system \eqref{1}, which is given at any equilibrium point $ E(S,I,C,A) $ by
\begin{equation}
\label{12b}
\begin{pmatrix}
-\mu-\dfrac{\partial f}{\partial S}I
& -\dfrac{\partial f}{\partial S}I-f(S,I) & 0  & 0 \\
\dfrac{\partial f}{\partial S}I
& \dfrac{\partial f}{\partial S}I+f(S,I) & \omega & \sigma  \\
0 & \phi & -\xi_{2} & 0  \\
0& \rho & 0 & - \xi_{3} \\
\end{pmatrix}.
\end{equation}
We recall that $E$ is locally asymptotically stable
if all the eigenvalues $\lambda $ of (\ref{12b})
satisfy the following condition \cite{Ahmed}:
$$
\vert arg(\lambda)\vert>\dfrac{\alpha \pi}{2}.
$$
From (\ref{12b}), the characteristic equation at $ E_{f} $ is given by
\begin{equation}
\label{13}
g(\lambda)=\lambda^{3}+a_{1}\lambda^{2}+a_{2}\lambda+a_{3}=0,
\end{equation}
where
\begin{align*}
a_{1}=&\xi_{1}+\xi_{2}+\xi_{3}-f(S_{0},0),&\\
a_{2}=&\xi_{1}\xi_{2}+\xi_{1}\xi_{3}+\xi_{2}\xi_{3}
-(\xi_{2}+\xi_{3})f(S_{0},0)-\phi\omega-\rho\sigma,\\
a_{3}=&(1-R_{0})\mathcal{D}.
\end{align*}
If $R_{0}>1$, than one has $ g(0)=a_{3}<0 $ and
$ \underset{\lambda\to+\infty}\lim g(\lambda)=+\infty$. Then,
there exists $ \lambda^{*}>0 $ satisfying $ g(\lambda^{*})=0 $.
In addition, we have $ \vert arg(\lambda^{*})\vert=0<\dfrac{\alpha \pi}{2} $.
Consequently, $ E_{f} $ is unstable when $ R_{0}>1$.

We define the Lyapunov functional $V_2$ for $E^*$ as follows:
\begin{multline*}
V_{2}(S,I,C,A)
=S-S^{*}-\displaystyle{\int_{S^{*}}^{S}}\frac{f(S^{*},I^{*})}{f(X,I^*)}dX
+I-I^{*}-I^{*}\ln\left(\dfrac{I}{I^*}\right)\\
+\dfrac{\omega}{\xi_2}\left(C-C^{*}-C^{*}\ln\left(\dfrac{C}{C^*}\right)\right)
+\dfrac{\sigma}{\xi_3}\left(A-A^{*}-A^{*}\ln\left(\dfrac{A}{A^*}\right)\right).
\end{multline*}
The fractional time derivative of $V_2$ along
the positive solutions of system \eqref{1} satisfies
\begin{equation*}
_{0}^{C}D_{t}^{\alpha}V_{2}
\leq \left(1-\dfrac{f(S^{*},I^{*})}{f(S,I^*)}\right)\, _{0}^{C}D_{t}^{\alpha}S
+\left(1-\dfrac{I^{*}}{I}\right)\, _{0}^{C}D_{t}^{\alpha}I
+\dfrac{\omega}{\xi_2}\left(1-\dfrac{C^{*}}{C}\right)\, _{0}^{C}D_{t}^{\alpha}C
+\dfrac{\sigma}{\xi_3}\left(
1-\dfrac{A^{*}}{A}\right)\, _{0}^{C}D_{t}^{\alpha}A.
\end{equation*}
Applying the equalities $\Lambda=\mu S^* +f(S^{*},I^{*})I^{*}$ and
$\xi_1 I^*=f(S^{*},I^{*})I^{*}+\omega C^* +\sigma A^*$, we get
\begin{align*}
_{0}^{C}D_{t}^{\alpha}V_{2}
&\leq \mu(S^{*}-S) \left(1-\dfrac{f(S^{*},I^{*})}{f(S,I^*)}\right)
+f(S^{*},I^{*})I^{*}\left(1-\dfrac{f(S^{*},I^{*})}{f(S,I^*)}\right)\\
&\quad +\dfrac{f(S^*,I^*)f(S,I)I}{f(S,I^*)}+\omega C^*\left(
1-\dfrac{CI^*}{C^*I}\right)+\sigma A^*\left(1-\dfrac{AI^*}{A^*I}\right)\\
&\quad +\dfrac{\omega \phi}{\xi_2}I^*\left(1-\dfrac{C^*I}{CI^*}\right)
+\dfrac{\sigma \rho}{\xi_3}I^*\left(1-\dfrac{A^*I}{AI^*}\right)-f(S,I)I_1-\xi_1I\\
&\quad  +\dfrac{\omega \phi I}{\xi_2}+\dfrac{\sigma \rho I}{\xi_3}+f(S^*,I^*)I^*\\
&\leq \mu(S^{*}-S) \left(1-\dfrac{f(S^{*},I^{*})}{f(S,I^*)}\right)
+\omega C^*\left(2-\dfrac{CI^*}{C^*I}-\dfrac{C^*I}{CI^*}\right)\\
&\quad+\sigma A^*\left(2-\dfrac{AI^*}{A^*I}-\dfrac{A^*I}{AI^*}\right)
+2f(S^*,I^*)I^*+\dfrac{f(S^*,I^*)f(S,I)I}{f(S,I^*)}\\
&\quad-\dfrac{f(S^*,I^*)^2I^*}{f(S,I^*)}-f(S^*,I^*)I-f(S,I)I^*\\
&\leq \mu(S^{*}-S) \left(1-\dfrac{f(S^{*},I^{*})}{f(S,I^*)}\right)
+\omega C^*\left(2-\dfrac{CI^*}{C^*I}-\dfrac{C^*I}{CI^*}\right)\\
&\quad+\sigma A^*\left(2-\dfrac{AI^*}{A^*I}-\dfrac{A^*I}{AI^*}\right)
+f(S^*,I^*)I^*\left[-1+\dfrac{f(S,I)I}{f(S,I^*)I^*}-\dfrac{I}{I^*}
+\dfrac{f(S,I^*)}{f(S,I)}\right]\\
& \quad+f(S^*,I^*)I^*\left[3-\dfrac{f(S^*,I^*)}{f(S,I^*)}
-\dfrac{f(S,I)}{f(S^*,I^*)}-\dfrac{f(S,I^*)}{f(S,I)}\right]\\
&\leq \mu(S^{*}-S) \left(1-\dfrac{f(S^{*},I^{*})}{f(S,I^*)}\right)
+\omega C^*\left(2-\dfrac{CI^*}{C^*I}-\dfrac{C^*I}{CI^*}\right)\\
&\quad+\sigma A^*\left(2-\dfrac{AI^*}{A^*I}-\dfrac{A^*I}{AI^*}\right)
+f(S^*,I^*)I^*\bigg(1-\dfrac{f(S,I)}{f(S,I^*)}\bigg)\bigg(
\dfrac{f(S,I^*)}{f(S,I)}-\dfrac{I}{I^*}\bigg)\\
&\quad +f(S^*,I^*)I^*\left[3-\dfrac{f(S^*,I^*)}{f(S,I^*)}
-\dfrac{f(S,I)}{f(S^*,I^*)}-\dfrac{f(S,I^*)}{f(S,I)}\right].
\end{align*}
Since the arithmetic mean is greater than or equal
to the geometric mean, it is clear that
\begin{equation*}
\begin{array}{l}
2-\dfrac{CI^*}{C^* I}-\dfrac{C^*I}{CI^* }\leq 0, \\[12 pt]
2-\dfrac{AI^*}{A^* I}-\dfrac{A^*I}{AI^* }\leq 0,\\[12 pt]
3-\dfrac{f(S^*,I^*)}{f(S,I^*)}-\dfrac{f(S,I)}{f(S^*,I^*)}
-\dfrac{f(S,I^*)}{f(S,I)}\leq 0,
\end{array}
\end{equation*}
and the equalities hold only for $S=S^*$, $I=I^*$, $C=C^*$ and $A=A^*$.
Note that
$$
1-\dfrac{f(S^*,I^*)}{f(S,I)} \geq 0 \quad \text{for}\quad S\geq S^*
$$
and
$$
1-\dfrac{f(S^*,I^*)}{f(S,I)} < 0 \quad \text{for}\quad S<S^*.
$$
This leads to
\begin{equation*}
(S^*-S)\left(1-\dfrac{f(S^*,I^*)}{f(S,I^*)}\right)\leq 0.
\end{equation*}
Therefore, $_{0}^{C}D_{t}^{\alpha}V_{2}\leq0$. Further, the largest invariant set in 
$\lbrace (S,I,C,A)\in \mathbb{R}^{4}:\, _{0}^{C}D_{t}^{\alpha}V_{2} \leq 0 \rbrace$
is the singleton $E^{*}$. The global stability of $ E^{*} $ follows
from LaSalle's invariance principle.
\hfil \qed
\end{proof}


\section{Fractional optimal control of the model}
\label{sec:optctrl}

In this section, our main aim is to minimize the number of HIV infected individuals and,
simultaneously, to reduce the cost associated with such strategies. This is achieved
by introducing public education into communities, as a preventive measure time dependent
control $v_{1}(t)$, to start ART treatment, and move $I$ individuals to the $C$ compartment,
while control $ v_{2}(t) $ is designed to provide effective treatment to infected individuals
with AIDS symptoms. Thus, we consider the following fractional optimal control problem:
\begin{equation}
\label{14}
\min J(I(t),v_{1}(t),v_{2}(t))=\int_{0}^{t_{f}}\left[
I(t)+A(t)+B_{1}\delta v_{1}^{2}(t)+B_{2}\delta v_{2}^{2}(t)\right] dt
\end{equation}
subject to the fractional control system
\begin{equation}
\label{eq:SICA_control}
\begin{cases}
_{0}^{C}D_{t}^{\alpha}S(t) = \Lambda  - \mu S(t)- f\left(S(t),I(t)\right)I(t),\\[0.2 cm]
_{0}^{C}D_{t}^{\alpha}I(t) = f\left(S(t),I(t)\right)I(t) - (\rho + v_{1}(t) + \mu)I(t)
+ v_{2}(t) A(t) + \omega C(t), \\[0.2 cm]
_{0}^{C}D_{t}^{\alpha}C(t) = v_{1}(t) I(t) - (\omega + \mu)C(t),\\[0.2 cm]
_{0}^{C}D_{t}^{\alpha}A(t) =  \rho \, I(t) - (v_{2}(t)+ \mu + d) A(t),
\end{cases}
\end{equation}
with given initial conditions
\begin{equation}
\label{eq:gic}
S(0)=S_0\geq0, \
I(0)=I_0\geq0, \
C(0)=C_0\geq0, \
A(0)=A_0\geq 0.
\end{equation}

\begin{remark}
While it is not necessary to use quadratic controls in the cost functional \eqref{14}, 
this is the most common way to penalize the use of controls: see, e.g., 
\cite{MR2316829} or \cite{MR3815138}.
\end{remark}

The parameters $0<B_{1}, B_{2}<\infty $ are positive weights,
$\delta$ is the maximum number of infectious individuals for
the problem without control, $ B_{i}\delta v_{i}^{2}(t) $, $ i=1,2 $,
is the cost of applying control effort $ v_{i} $, and $ t_{f} $
is the duration of the control program. The set of admissible control functions is
\begin{equation}
\label{omega:set}
\mathcal{U}=\left\lbrace (v_{1}(\cdot),v_{2}(\cdot))
\in L^{\infty}(0,t_{f}) : 0 \leq v_{i}(t)
\leq {v_i}_{\max} \leq 1,\: i=1,2,\: \forall t
\in [0,t_{f}]\right\rbrace.
\end{equation}
To obtain the necessary optimality conditions for our fractional
optimal control problem, we define the Hamiltonian function as follows:
\begin{equation}
\begin{split}
H&=I+A+B_{1}\delta v_{1}^{2}(t)+B_{2}\delta v_{2}^{2}(t)
+\xi_{1}\left(  \Lambda  - \mu S(t)- f\left(S(t),I(t)\right)I(t)\right)\\
&\quad+\xi_{2}\left( f\left(S(t),I(t)\right)I(t) - (\rho + v_{1}(t) + \mu)I(t)
+ v_{2}(t) A(t) + \omega C(t)\right)\\
&\quad +\xi_{3}\left(  v_{1}(t) I(t) - (\omega + \mu)C(t)\right)
+\xi_{4}\left( \rho \, I(t) - (v_{2}(t)+ \mu + d) A(t)\right).
\end{split}
\end{equation}
Applying \cite[Theorem 4.1 and Lemma 4.2]{Kheiri}, the necessary conditions
for the optimality of \eqref{14} are given by \eqref{eq:SICA_control}
supplemented with the adjoint system
\begin{equation}
\label{adj2:system}
\begin{cases}
_{0}^{C}D_{t}^{\alpha}\xi_{1}(t')
= - \mu \xi_{1}(t')+ \dfrac{\partial f}{\partial S}I(t')(\xi_{2}(t')-\xi_{1}(t')),\\[0.2 cm]
_{0}^{C}D_{t}^{\alpha}\xi_{2}(t')
= 1-\left( \dfrac{\partial f}{\partial I}I(t')+f(S,I)\right)\xi_{1}(t')
+ \xi_{3}(t')v_{1}(t') +\rho\xi_{4}(t')\\[0.2 cm]
\qquad +\left( \dfrac{\partial f}{\partial I}I(t')
+f(S,I)- \rho - v_{1}(t')-\mu\right) \xi_{2}(t'),\\[0.2 cm]
_{0}^{C}D_{t}^{\alpha}\xi_{3}(t') = \omega\xi_{2}(t')
- (\omega + \mu)\xi_{3}(t'),\\[0.2 cm]
_{0}^{C}D_{t}^{\alpha}\xi_{4}(t') =  1+v_{2}(t')\xi_{2}(t')
- (v_{2}(t') + \mu + d) \xi_{4}(t'),
\end{cases}
\end{equation}
with $t'=t_{f}-t$, the initial and transversality conditions
\begin{equation}
\label{ini:cond}
\begin{gathered}
S(0)=S_0\geq0, \
I(0)=I_0\geq0, \
C(0)=C_0\geq0, \
A(0)=A_0\geq 0,\\
\xi_{1}(t_{f})=\xi_{2}(t_{f})=\xi_{3}(t_{f})=\xi_{4}(t_{f})=0.
\end{gathered}
\end{equation}
Furthermore, the optimal controls $ v_{1}^{*} $
and $ v_{2}^{*} $ are given by
\begin{equation}
\label{optim:cond}
\begin{split}
v_{1}^{*}&=\min\left( {v_1}_{\max},\max\left( 0,
\dfrac{(\xi_{2}-\xi_{3})I}{2B_{1}\delta}\right) \right),\\
v_{2}^{*}&=\min\left( {v_2}_{\max},\max\left( 0,
\dfrac{(\xi_{4}-\xi_{2})A}{2B_{2}\delta}\right) \right) .
\end{split}
\end{equation}


\section{Applications and numerical simulations}
\label{sec:5}

The Pontryagin Maximum Principle is used to numerically
solve the optimal control problem \eqref{14}--\eqref{omega:set},
as discussed in Section~\ref{sec:optctrl}, in the classical ($\alpha = 1$)
and fractional ($\alpha<1$) cases, using the
predict-evaluate-correct-evaluate (PECE) method of Adams--Basforth--Moulton
\cite{diethelm2005algorithms} implemented in MATLAB.
First, we solve system \eqref{eq:SICA_control} by the PECE procedure
with initial values for the state variables based on Moroccan data \cite{lotfi}:
\[
S_0=(N_0-(2+9))/N_0,\quad I_0=2/N_0, \quad C_0=0,\quad A_0=9/N_0,
\]
with $N_0$=23023935 and a guess for the control over the time interval $[0,t_f]$,
thereby obtaining  the values of the state variables $S$, $I$, $C$ and $A$.
As in \cite{Rosa}, a change of variable is applied to the adjoint
system \eqref{adj2:system} and to the transversality conditions,
obtaining the  fractional initial value problem
\eqref{adj2:system}--\eqref{ini:cond}. Such IVP
is also solved with the PECE procedure, and
the values of the co-state variables $\xi_i$, $i=1,\ldots,4$, are obtained.
The controls are then updated by a convex combination of the previous controls
and the current values computed according to \eqref{optim:cond}. This procedure
is repeated iteratively until the values of all the variables and the values
of the controls are very close to the ones of the previous iteration.
The solutions of the classical model were successfully confirmed
by a classical forward-backward scheme, also implemented in MATLAB.

Solving the initial system \eqref{1} with $\alpha=1$ (classical derivatives),
we notice that the maximum number of $I$ individuals, $\delta$,
is $1.24\times 10^{-7}$. This value is the one we use in numerical experiments.
We also use ${v_1}_{\max}={v_2}_{\max} = 1$, $B_1=B_2=2.5$, and the other
parameters are fixed according to Table~\ref{tab:param} (see \cite{lotfi}),
where $ \beta $ is the effective transmission rate.
\begin{center}
\captionof{table}{Parameter values of system \eqref{1}.}\label{tab:param}
\begin{tabular}{ccc} \hline
parameter & description & value   \\ \hline
$\mu$ & Natural death rate & 1/74.02   \\		
$\Lambda$ & Recruitment rate & 2.19$\mu$  \\		
$\beta$ & HIV transmission rate & 0.755 \\
$\phi$ & HIV treatment rate for $ I $ individuals & 1\\
$\rho$ & Default treatment rate for $ I $ individuals & 0.1 \\
 $ \sigma $ & AIDS treatment rate & 0.33\\
 $\omega$ & Default treatment rate for $ C $ individuals & 0.09\\
$ d $ & AIDS induced death rate & 1\\
\hline
\end{tabular}
\end{center}
Because the World Health Organization (WHO)
goals for most diseases are usually fixed
for five years periods, we considered $t_f = 5$.

Without loss of generality, in what follows
we consider the incidence function to be
\[
f(S,I)=\beta S.
\]
This function is chosen because, when compared with other incidence functions,
this was the one that better fitted to real data \cite{lotfi}.

\begin{figure}[!htb]
\centering
\begin{subfigure}[b]{0.46\textwidth}\centering
\includegraphics[scale=0.46]{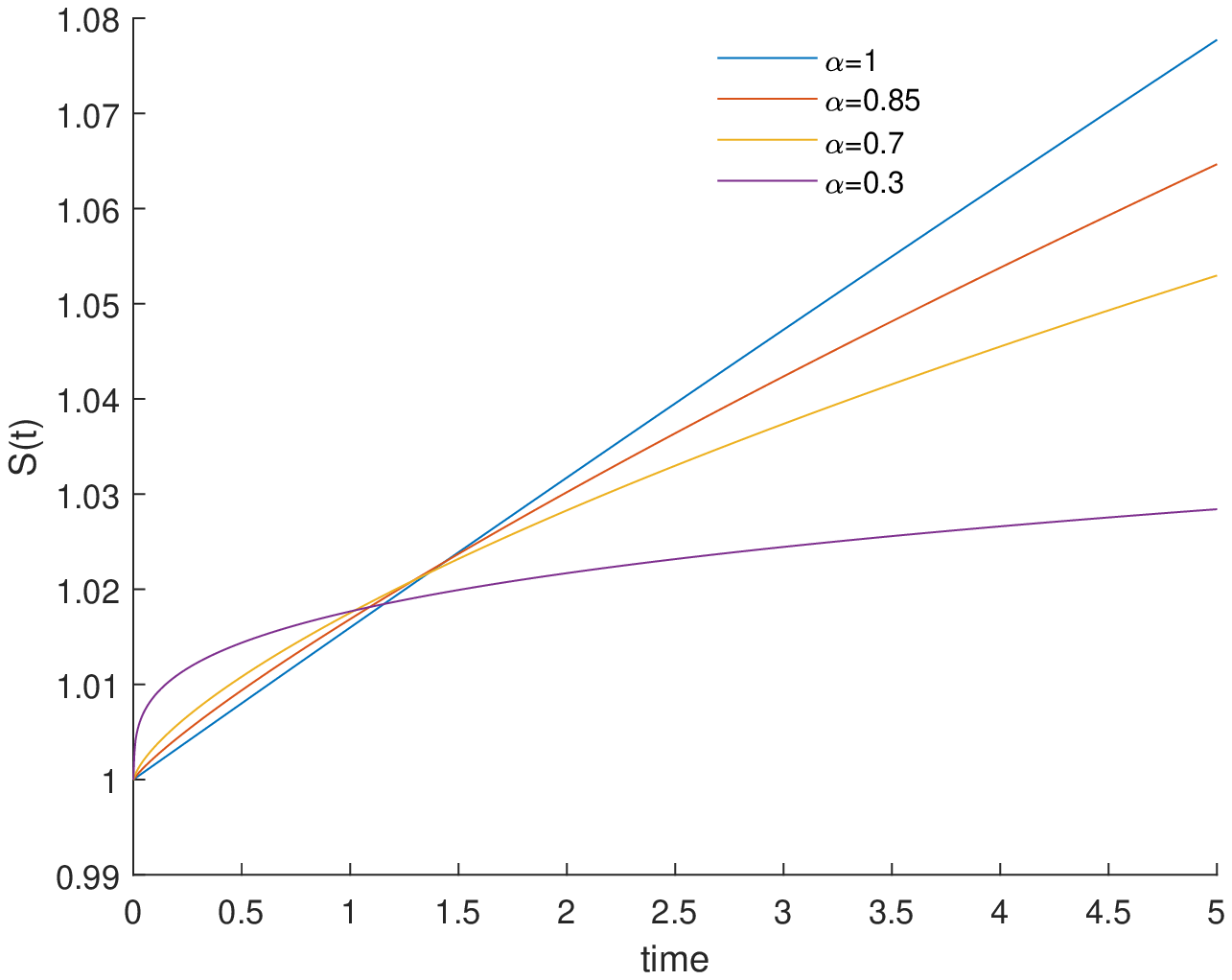}
\caption{Susceptible individuals.}
\end{subfigure}\hspace*{1cm}
\begin{subfigure}[b]{0.46\textwidth}
\centering
\includegraphics[scale=0.46]{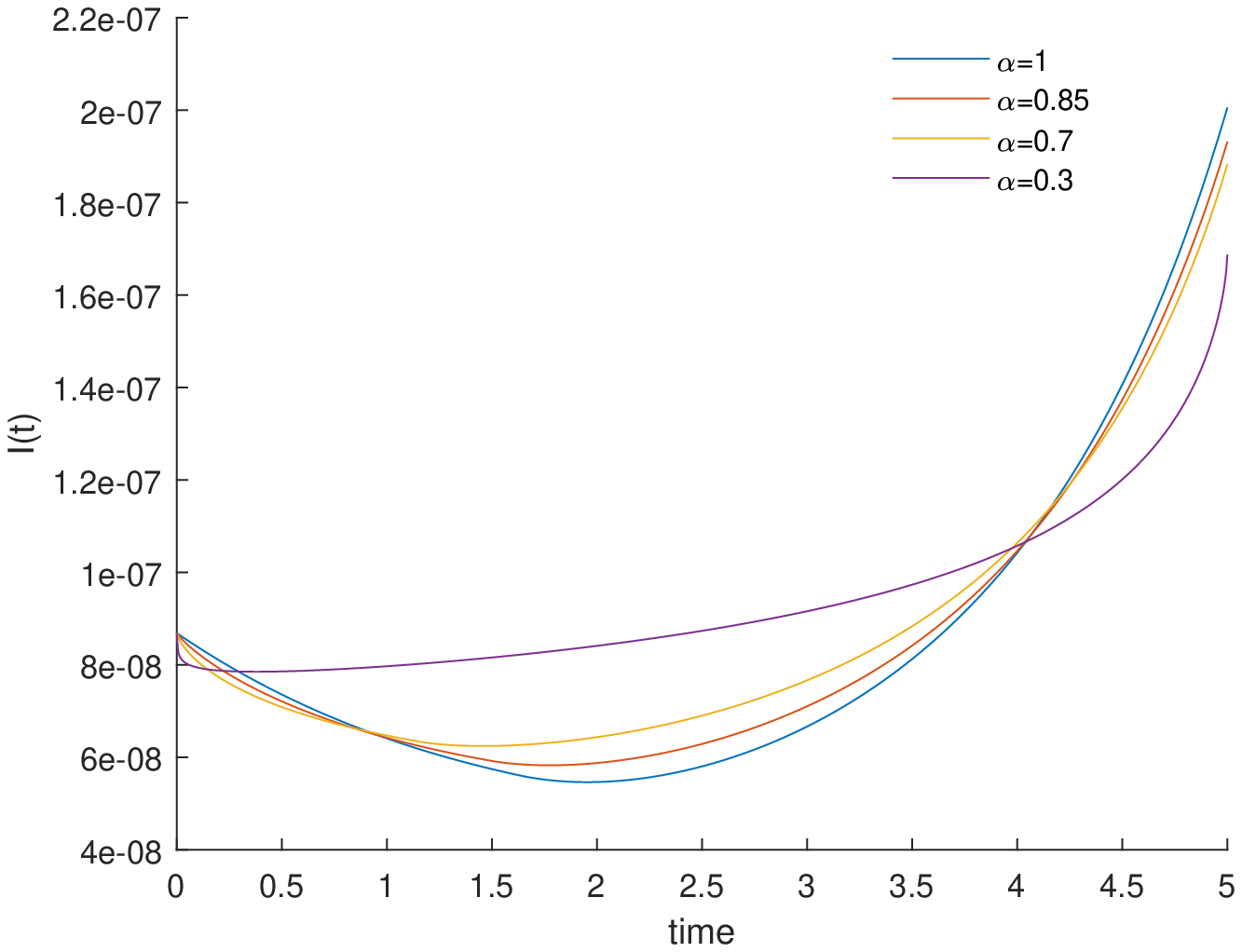}
\caption{HIV infected individuals.}
\end{subfigure}\\
\begin{subfigure}[b]{0.46\textwidth}
\centering
\includegraphics[scale=0.46]{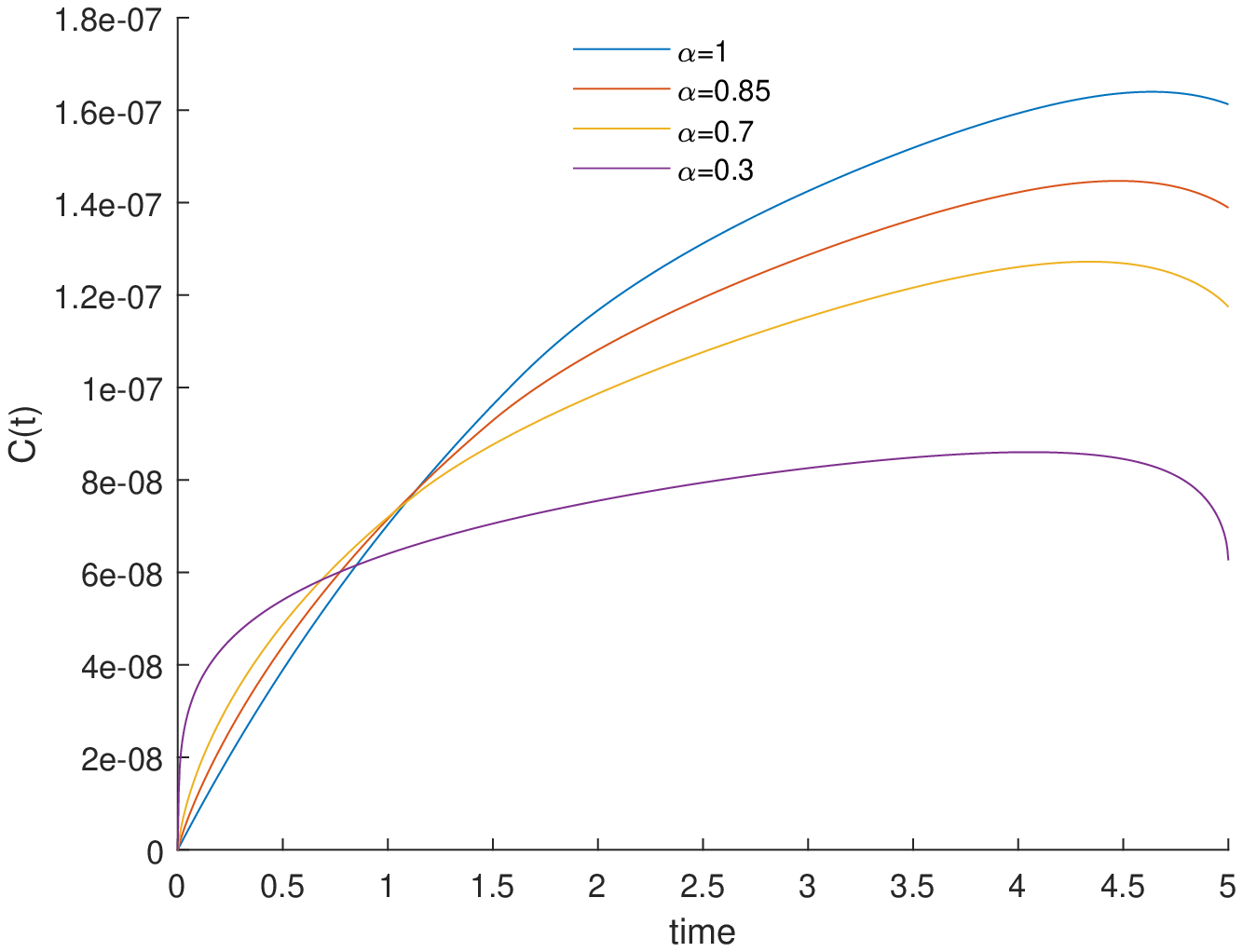}
\caption{HIV infected ind. under ART
treatment}\label{fig:C_sev:alphas}
\end{subfigure}\hspace*{1cm}
\begin{subfigure}[b]{0.46\textwidth}
\centering
\includegraphics[scale=0.46]{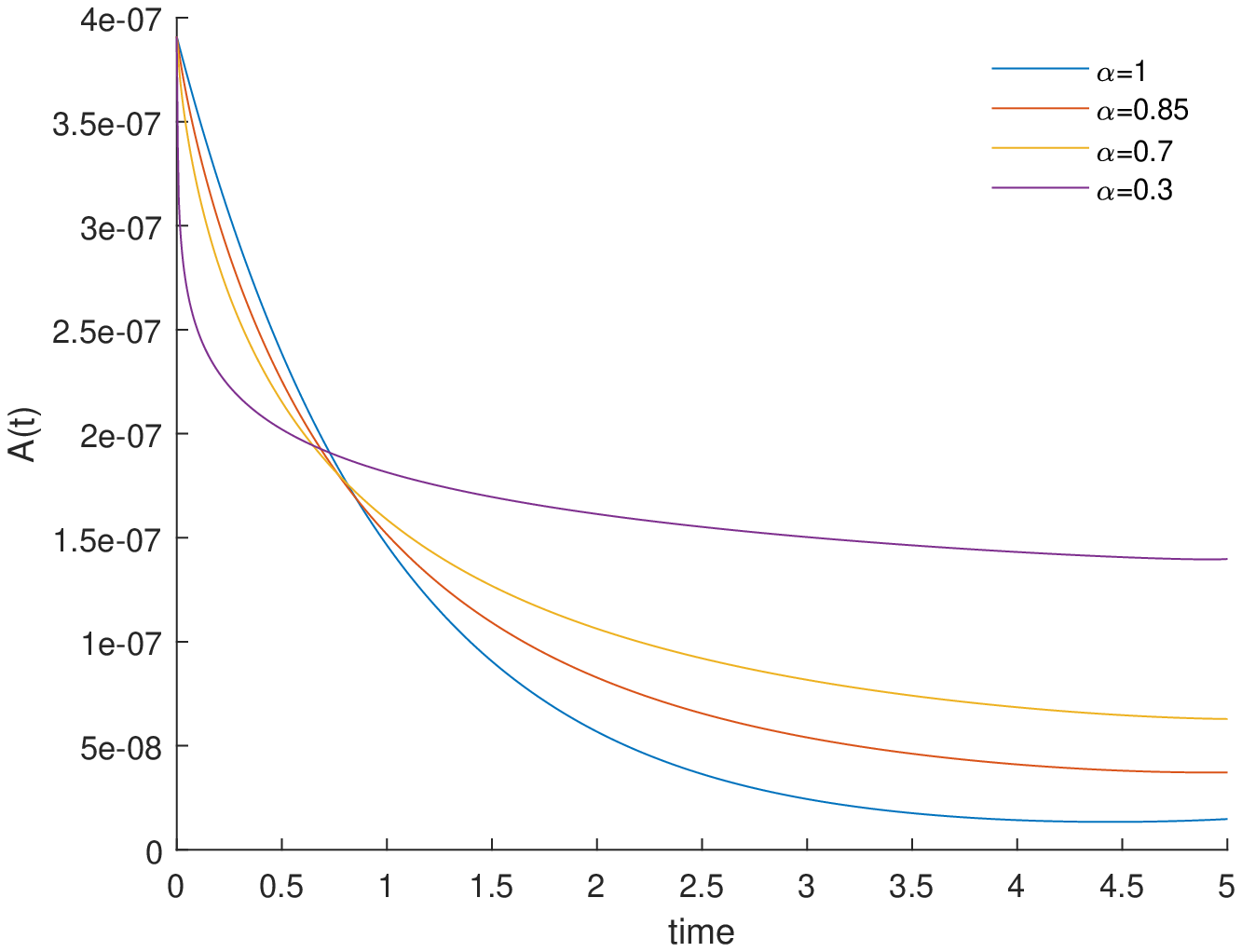}
\caption{ HIV infected individuals with AIDS
symptoms.}
\end{subfigure}
\caption{State variables of the FOCP \eqref{14}--\eqref{optim:cond},
with values from Table~\ref{tab:param}, weights $B_1=B_2=2.5$,
and the fractional order derivatives
$\alpha=1.0,$ $0.85$, $0.7$ and $0.3$.}
\label{fig:states_var:alphas}
\end{figure}

In order to perceive the effect of the derivative order on the variation
of the variables of the problem, we considered, as in \cite{TB2019},
some fractional order derivatives, namely $\alpha=1.0$, $0.85$, 
$0.70$ and $0.3$. In Figure~\ref{fig:states_var:alphas}, \ref{fig:v1_var:alphas},
and \ref{fig:v2_var:alphas}, we have the solutions of the fractional
optimal control problem (FOCP) for that values of $\alpha$. We observe
that a change in the derivative order corresponds to variations of the
state variables and of the first control, $v_1$. On the other hand,
the second control, $v_2$, does not vary with that change, remaining null.
Treatment of AIDS individuals was considered cheaper, i.e., smaller values
for $B_2$ (weight of $v_2$ in the cost functional) were considered
but $v_2$ has not changed. This means that treating people with AIDS symptoms
is useless when we can act over infected people with ART treatment.

The existence of an endemic situation ($R_0=7.534 > 1$ \cite{lotfi}),
the existence of a high percentage of susceptible individuals and
a control that vanishes at the end of the time interval motivates that,
in the end of the time interval, the number of $I$ individuals exceeds
its initial value. Other values of $\alpha$, lower than one, were also
tested, but the results do not changed qualitatively.  According with
Figure~\ref{fig:states_var:alphas}, decreasing the derivative order,
$\alpha$, means that, after a certain value of time, it decreases
the number of individuals of compartments $S$ and $C$ while increasing
the value of individuals in compartment $A$. We note that the variation
of $\alpha$ has little impact in the variation of infected individuals, $I$.

\begin{remark}
Solutions of adjoint variables can be easily included in numerical simulations:
see, e.g., \cite{MR3771538}. While in \cite{MR3771538} the inclusion of the adjoint 
variables is important, because of the non-regularity of the control variable, 
which is obviously explained by those variables, here, however, the inclusion 
of the adjoint variables do not bring new insights:
the nonzero control $v_1$ has a classical evolution, starting at its maximum, one,  
after it decreases and vanishes at the end of the time interval. This behaviour is common 
to many known examples and the adjoint variables, in turn, follow qualitatively this evolution, 
with different magnitudes.
\end{remark}

\begin{figure}[!htb]
\centering
\begin{subfigure}[b]{0.46\textwidth}\centering
\includegraphics[scale=0.46]{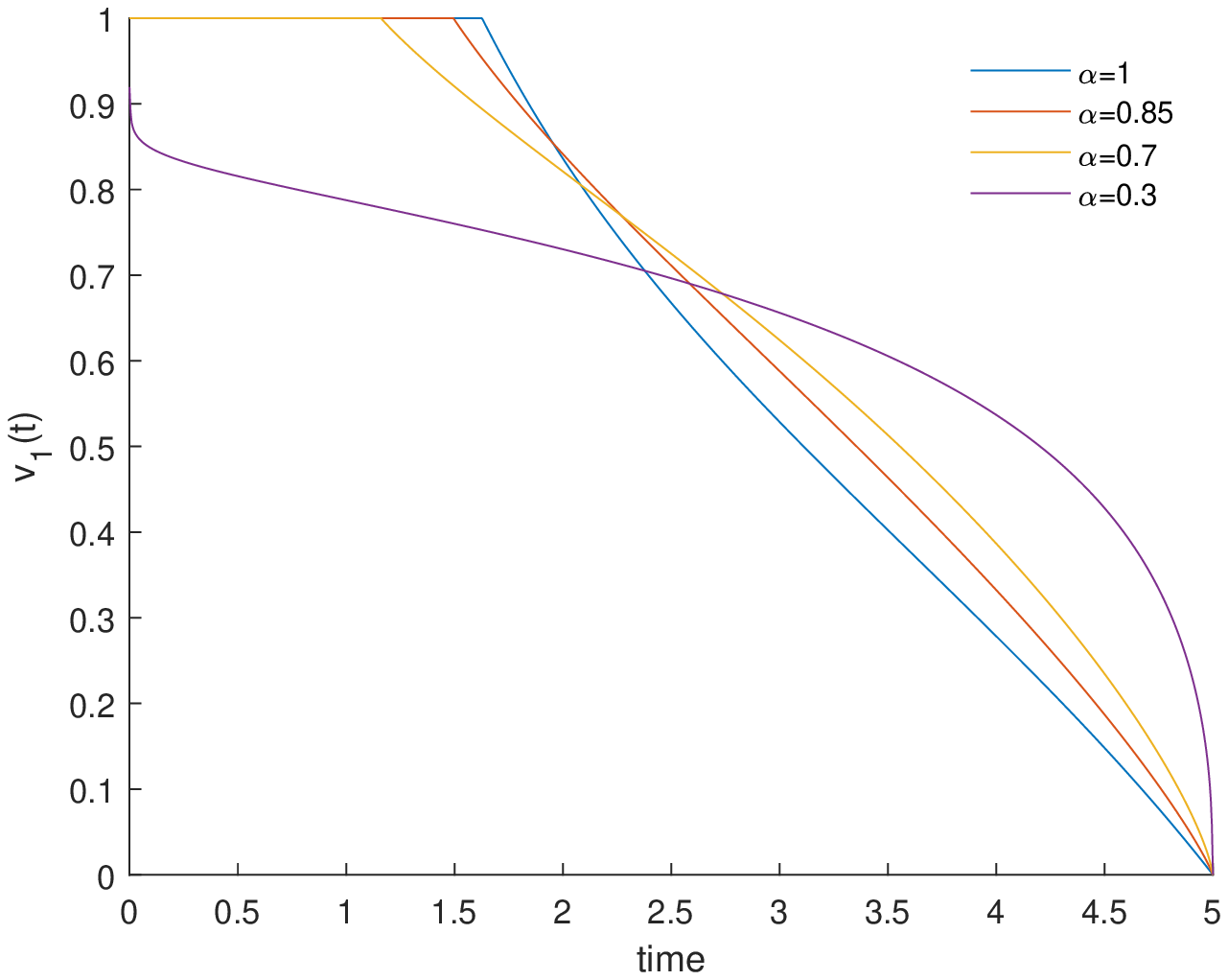}
\caption{Optimal control $v_1$.}\label{fig:v1_var:alphas}
\end{subfigure}\hspace*{1cm}
\begin{subfigure}[b]{0.46\textwidth}
\centering
\includegraphics[scale=0.46]{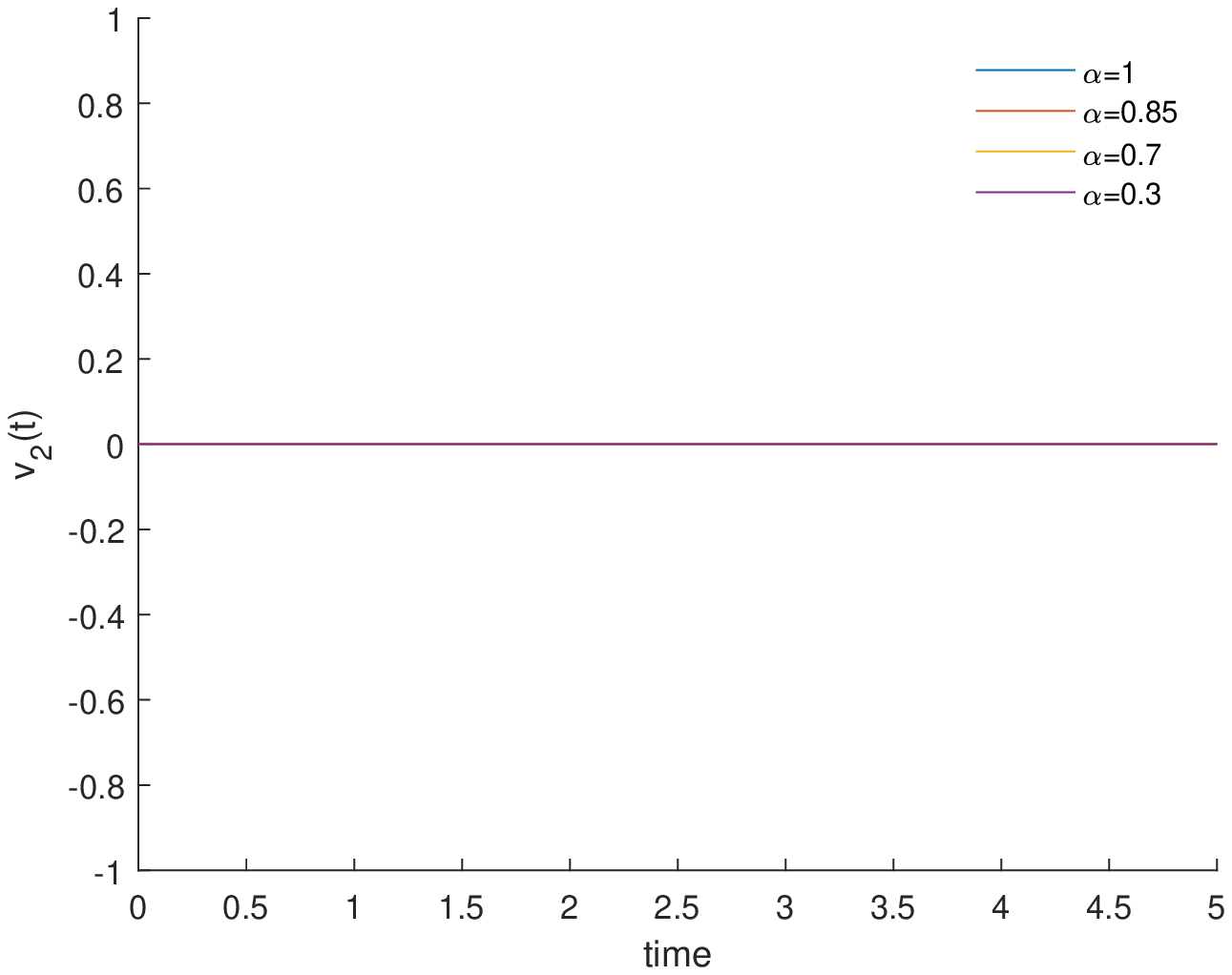}
\caption{Optimal control $v_2$.}\label{fig:v2_var:alphas}
\end{subfigure}\\
\begin{subfigure}[b]{0.46\textwidth}
\centering
\includegraphics[scale=0.46]{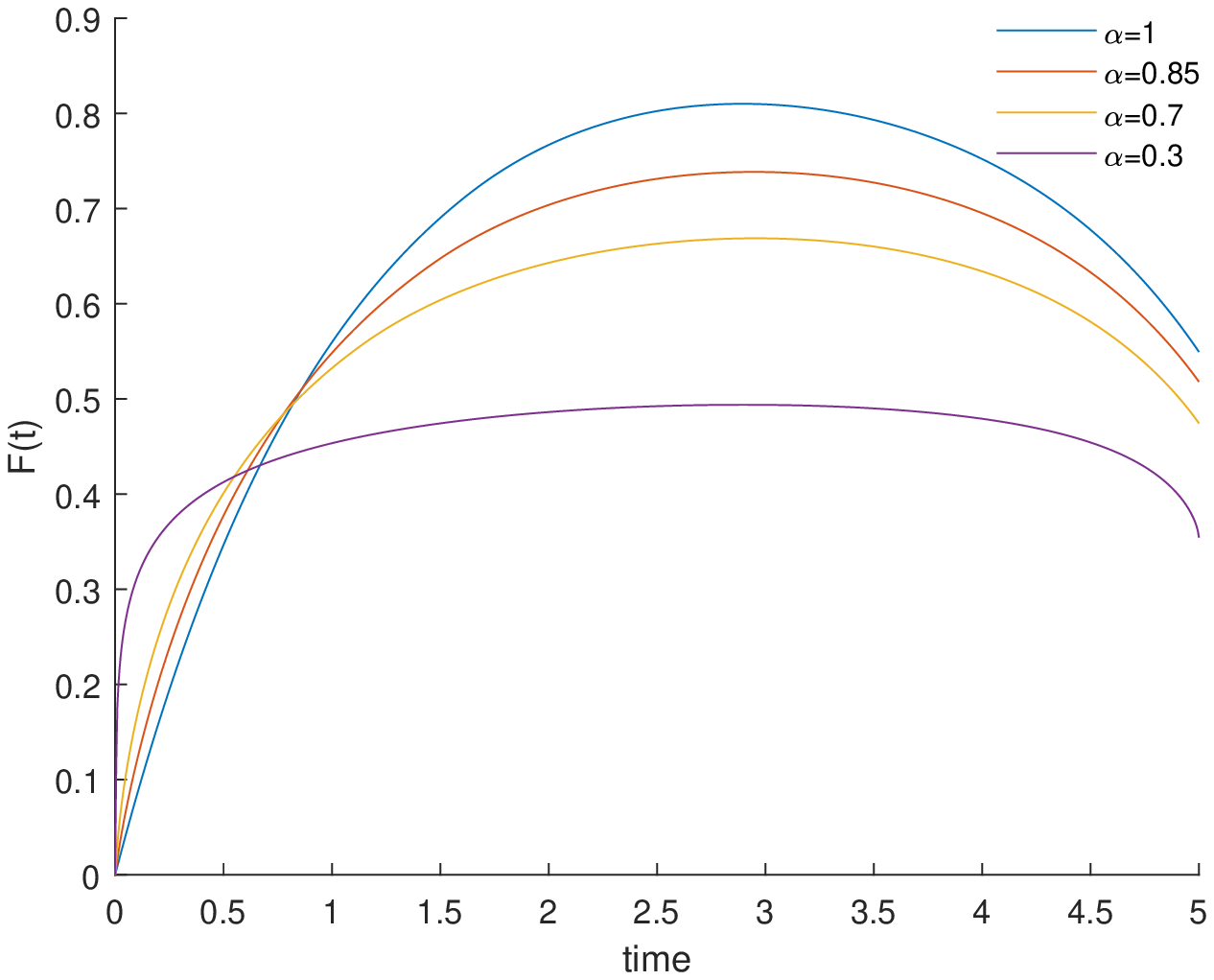}
\caption{Efficacy function $F(t)$.}\label{fig:efficacy:alphas}
\end{subfigure}
\caption{Control functions $v_1(t)$ and $v_2(t)$ associated
to the FOCP \eqref{14}--\eqref{optim:cond} with values from
Table~\ref{tab:param}, weights $B_1=B_2=2.5$, and the fractional 
order derivatives $\alpha=1.0,$ $0.85$, $0.7$ and $0.3$.}
\label{fig:v1_v2:alphas}
\end{figure}

The efficacy function \cite{rodrigues2014cost} is defined by
\begin{equation}
\label{efficacy_function}
F(t)=\frac{i(0)-i^*(t)}{i(0)}=1-\frac{i^*(t)}{i(0)},
\end{equation}
where $i^*(t)=I^*(t)+A^*(t)$ is the optimal solution of the fractional optimal
control and $i(0)=I(0)+A(0)$ is the corresponding initial condition. This
function measures the proportional variation in the number of infected individuals,
HIV infected or infected with AIDS, after the application of the controls
$\{v_1^*,v_2^*\}$, by comparing the number of infectious individuals
at time $t$ with the initial value $i(0)$. In Figure~\ref{fig:efficacy:alphas},
the efficacy function is exhibited for the three considered values of $\alpha$.
Interestingly, the classical model is the most effective.

\begin{figure}[!htb]
\centering
\begin{subfigure}[b]{0.46\textwidth}\centering
\includegraphics[scale=0.46]{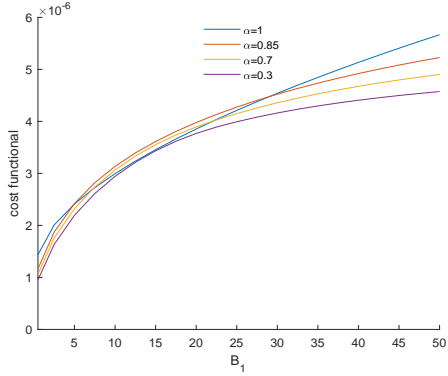}
\caption{Evolution of the cost functional $J$ \eqref{14}.}\label{fig:evol_J}
\end{subfigure}\hspace*{1cm}
\begin{subfigure}[b]{0.46\textwidth}
\centering
\includegraphics[scale=0.46]{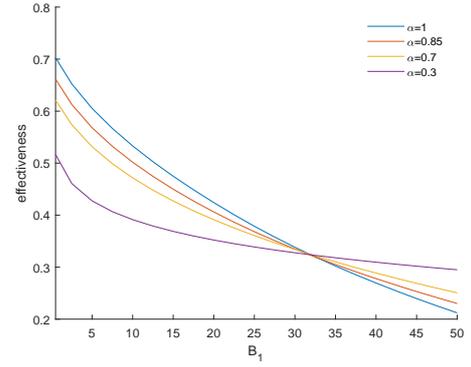}
\caption{Evolution of the effectiveness $\overline F$ \eqref{eq:F}.}\label{fig:evol_F}
\end{subfigure}
\caption{Impact of the variation of the weight $B_1$ on the cost functional value
$J$ (left) and on the effectiveness measure $\overline{F}$ (right)
for the fractional order derivatives $\alpha=1.0,$ $0.85,$ $0.7$ and $0.3$.}
\label{fig:evol_J_F}
\end{figure}

Some summary measures are presented to evaluate the cost
and the effectiveness of the proposed fractional
control measures during the intervention period. The
total cases averted by the intervention during
the time period $t_f$ is defined in \cite{rodrigues2014cost} by
\begin{equation}
\label{eq:A}
AV=i(0)i_f-\int_0^{t_f}i^*(t)~dt,
\end{equation}
where $i^*(t)=I^*(t)+A^*(t)$ is the optimal solution corresponding
to the fractional optimal controls $\{v_1^*, v_2^*\}$
and $i(0)=I(0)+A(0)$ is the corresponding initial condition.

Effectiveness is defined as the proportion of cases averted
on the total cases possible under no intervention \cite{rodrigues2014cost}:
\begin{equation}
\label{eq:F}
\overline{F}=\frac{AV}{ i(0)t_f}
=1-\frac{\displaystyle \int_0^{t_f}i^*(t)~dt}{i(0)t_f}.
\end{equation}
The total cost associated with the intervention
is defined in  \cite{rodrigues2014cost} by
\begin{equation}
\label{eq:TCI}
TC=\int_0^{t_f} C_1 \, v_1^*(t)I^*(t)+C_2 \, v_2^*(t)A^*(t)~dt,
\end{equation}
where $C_i$ corresponds to the per person unit cost
of the two possible interventions: (i)~detection
and treatment of HIV infected individuals ($C_1$);
(ii)~and detection and treatment of  HIV infected
individuals with AIDS ($C_2$).
Following \cite{rodrigues2014cost},
the average cost-effectiveness ratio is given by
\begin{equation}
\label{eq:ACER}
ACER=\frac{TC}{AV}.
\end{equation}

\begin{table}[!htb]
\centering
\caption{Summary of cost-effectiveness measures for classical ($\alpha = 1$)
and fractional ($0<\alpha <1$) HIV-AIDS disease optimal control problem. 
Parameters according to Table~\ref{tab:param} and $C_1=C_2=1.$}\label{tab:efficacy}
\begin{tabular}{c@{\hspace*{1cm}}c@{\hspace*{1cm}}c@{\hspace*{1cm}}c@{\hspace*{1cm}}c}
\toprule
$\alpha$ & $AV$  & $TC$ & $ACER$ & $\overline{F}$  \\[1mm] \midrule
 1.0  &  1.55815e-06  &  2.21065e-07 &    0.141877  &  0.652268\\
 0.85 &   1.46382e-06  &  2.37665e-07 &    0.16236  &  0.612779\\
0.70  &  1.36489e-06  &  2.54213e-07  &   0.186252  &  0.571365\\
0.30 &   1.09485e-06  &  2.94279e-07  &   0.268786  &  0.458322\\
 \bottomrule
\end{tabular}
\end{table}

In Table~\ref{tab:efficacy}, the cost-effectiveness measures,
for our fractional optimal control problem, are summarized.
Those results show the effectiveness of the control $v_1$
to reduce HIV infectious individuals and the superiority
of the classical model ($\alpha=1$).

The impact of the variation of the weight $B_1$ over the cost
functional is displayed in Figure~\ref{fig:evol_J}, and
over the effectiveness measure $\overline{F}$ is presented
in Figure~\ref{fig:evol_F}. When the cost of treatment increases,
we observe that: (i) the fractional model can be more effective
in reduction of $I+A$ individuals; (ii) the cost functional also
increases and the classical model is the one with lower values.

The  fractional model is more effective when treatment is expensive,
i.e., hard to implement. To illustrate this behaviour, we considered
$B_1=40$ and determined the derivative order with the best effectiveness
measure. The highest value of effectiveness, $0.309505$, was attained
with $\alpha=0.30$, a quite low value. The respective value for the
classical model ($\alpha=1$) is $0.269456$.  Figures~\ref{fig:states_alpha 023}
and \ref{fig:v1__F:023} compare the fractional solution with the classical one.
When compared with the classical model, the variables $S$ and $C$,
of the fractional model, behave analogously to above cases with $B_1=2.5$
and other values of $\alpha$, but move further apart. With respect to $I+A$
individuals, we see that, in average, the fractional solution is lower than
the classical solution. We also notice that the first control of the fractional
model is more intense than the one of the classical model in most part of the 
time interval. Such behaviour of control $v_1$ can be the reason of the higher 
effectiveness of such model. The second control is not exhibited because it remains null.
The efficacy function, presented in Figure~\ref{fig:efficacy:023}, confirms
that the fractional model is the best choice in this case.
\begin{figure}[!htb]
\centering
\begin{subfigure}[b]{0.46\textwidth}\centering
\includegraphics[scale=0.46]{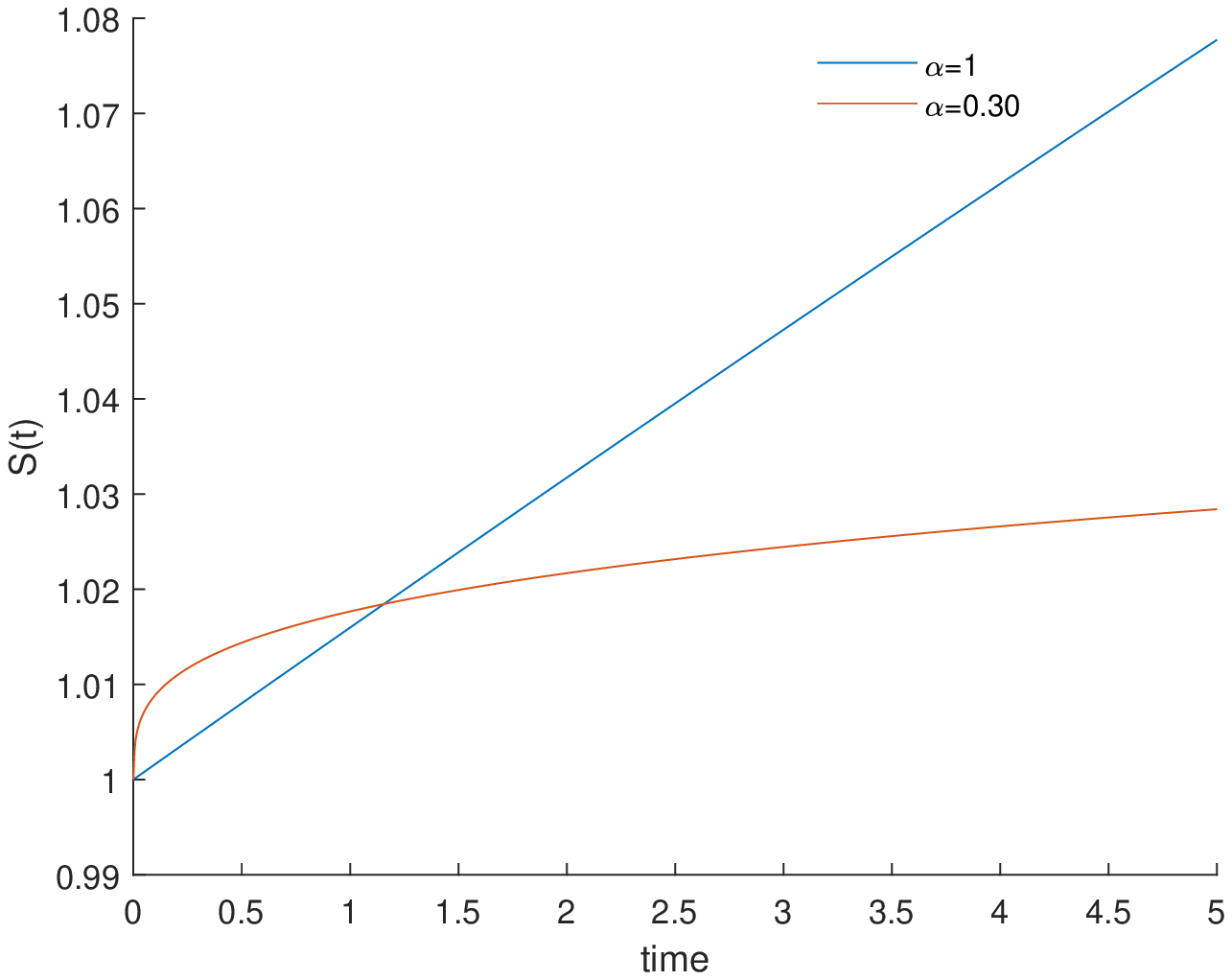}
\caption{Susceptible individuals.}
\end{subfigure}\hspace*{1cm}
\begin{subfigure}[b]{0.46\textwidth}
\centering
\includegraphics[scale=0.46]{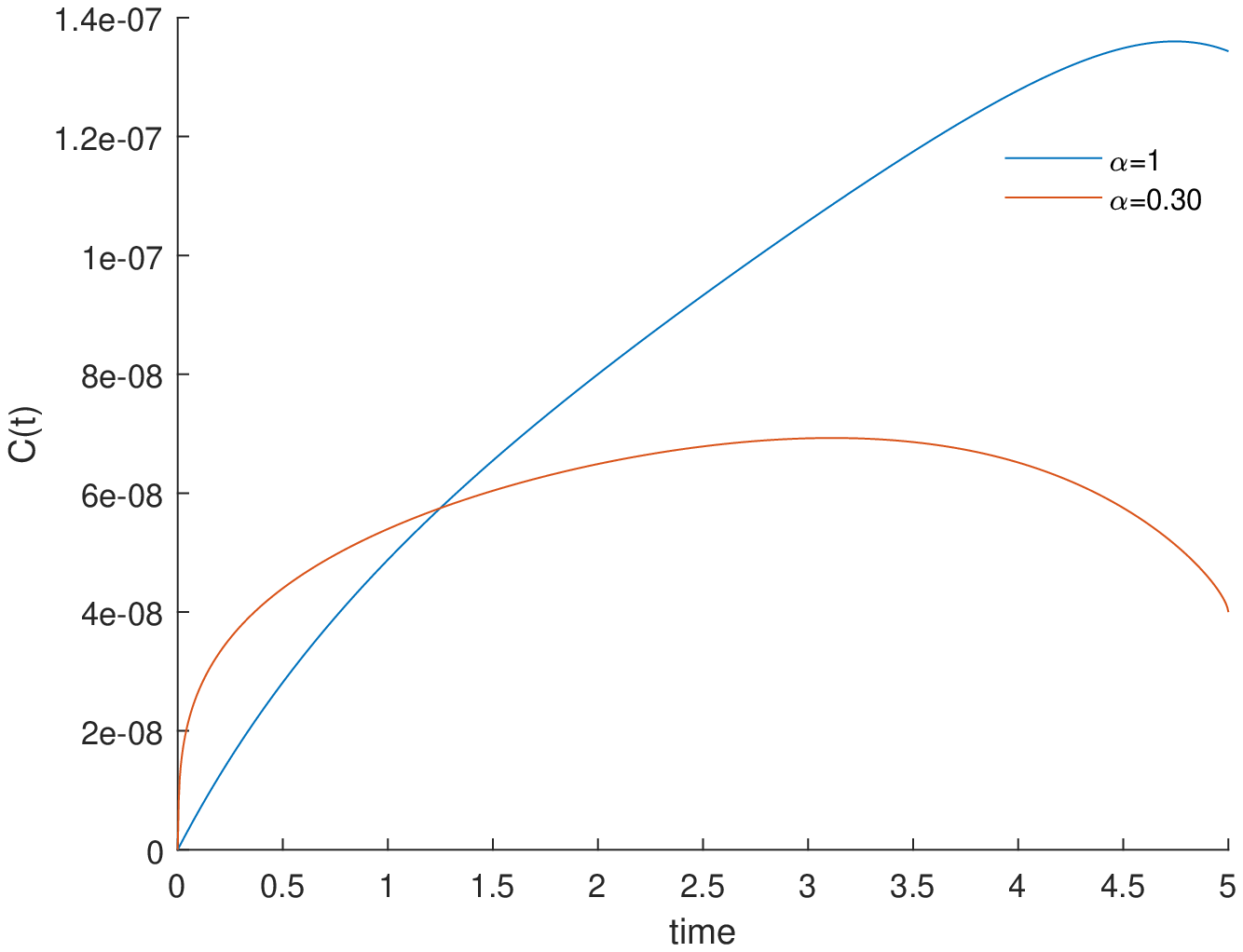}
\caption{HIV infected ind. under ART
treatment}\label{fig:C_023}\end{subfigure}\\
\begin{subfigure}[b]{0.46\textwidth}
\centering
\includegraphics[scale=0.46]{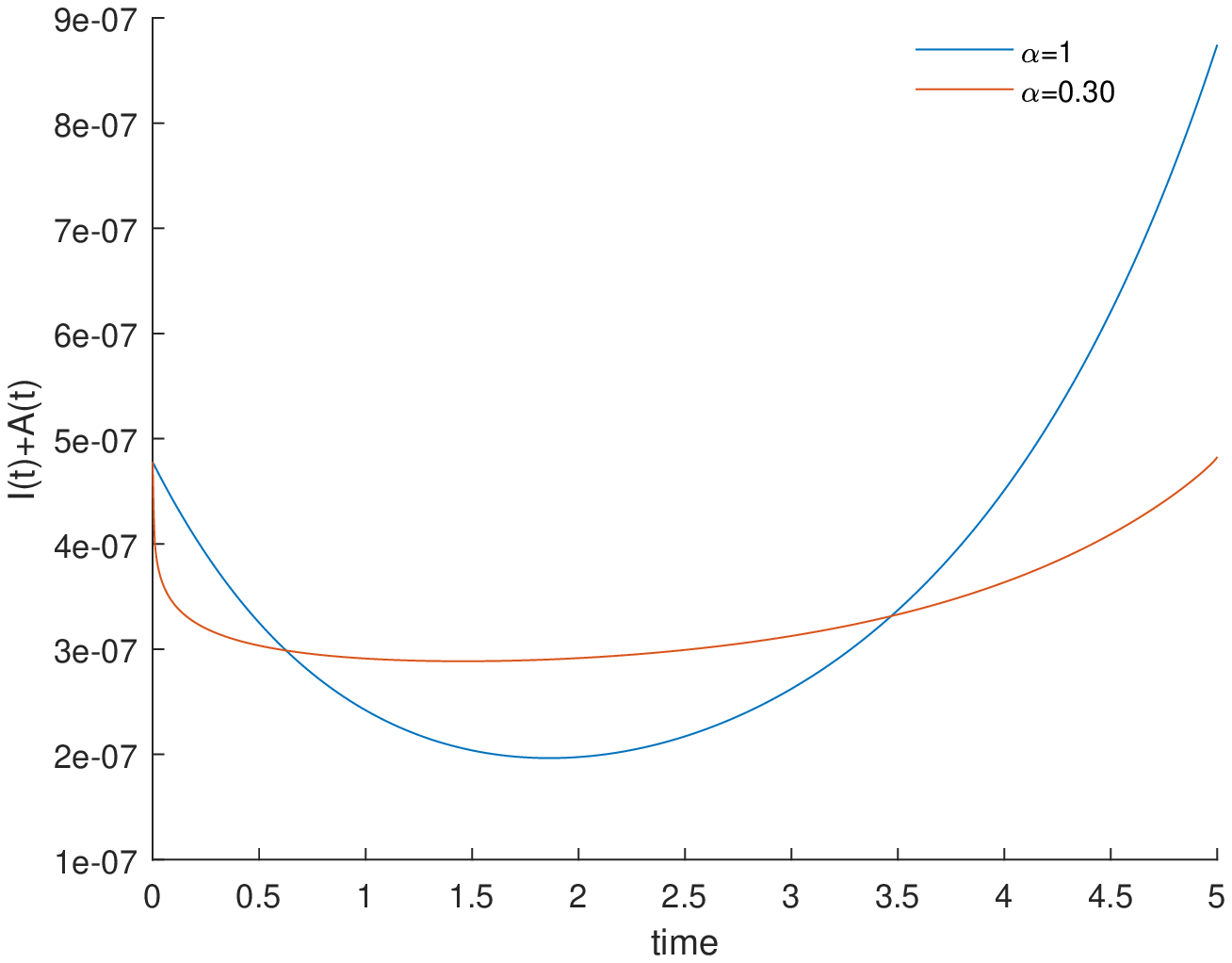}
\caption{HIV infected and HIV infected with AIDS symptoms individuals.}
\end{subfigure}
\caption{State variables of the FOCP \eqref{14}--\eqref{optim:cond},
with values from Table~\ref{tab:param}, weights $B_1=B_2=40$,
and the fractional order derivatives $\alpha=1.0$ and $0.30$.}
\label{fig:states_alpha 023}
\end{figure}

\begin{figure}[!htb]
\centering
\begin{subfigure}[b]{0.46\textwidth}\centering
\includegraphics[scale=0.46]{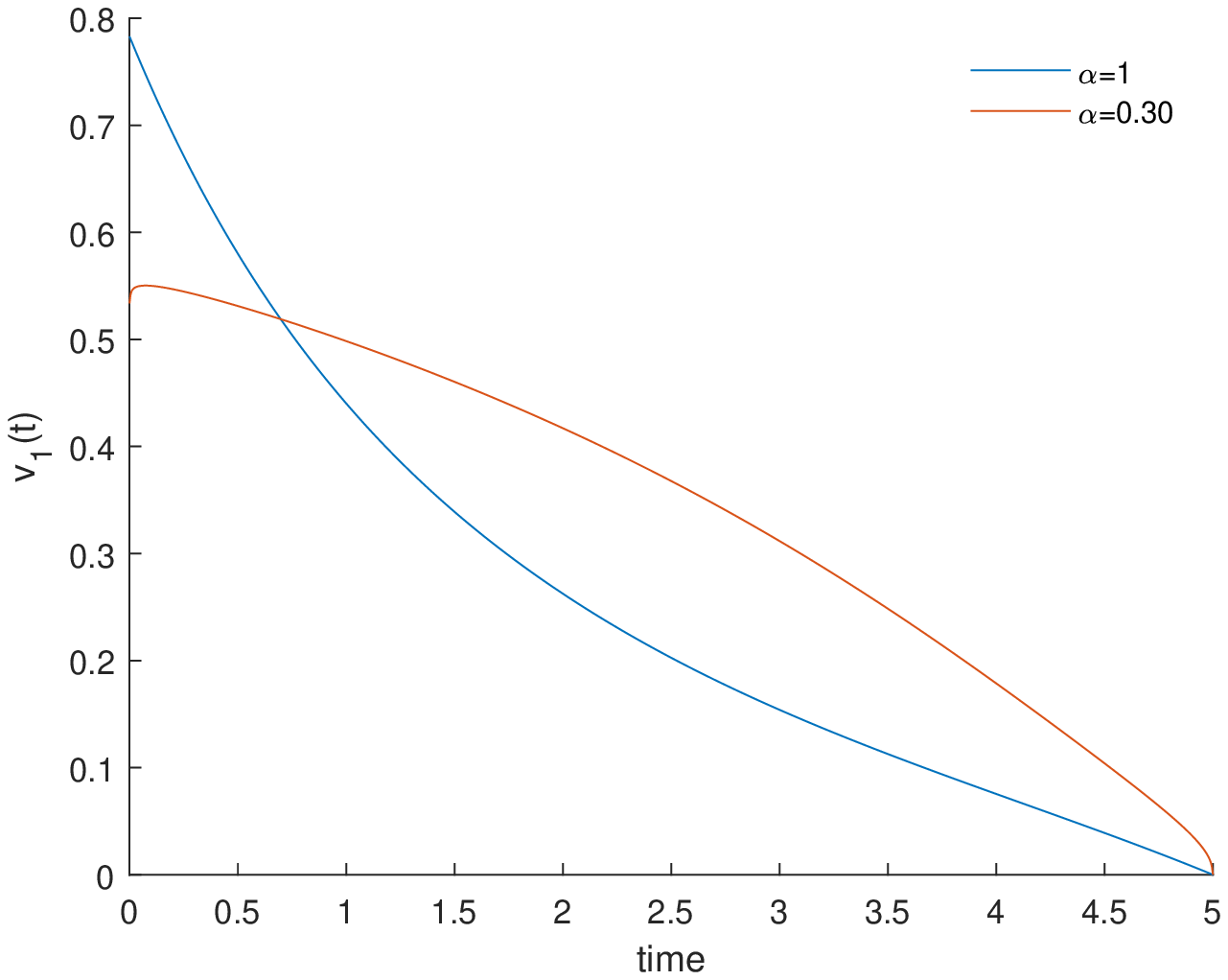}
\caption{Optimal control $v_1$.}\label{fig:v1_alpha023}
\end{subfigure}\hspace*{1cm}
\begin{subfigure}[b]{0.46\textwidth}
\centering
\includegraphics[scale=0.46]{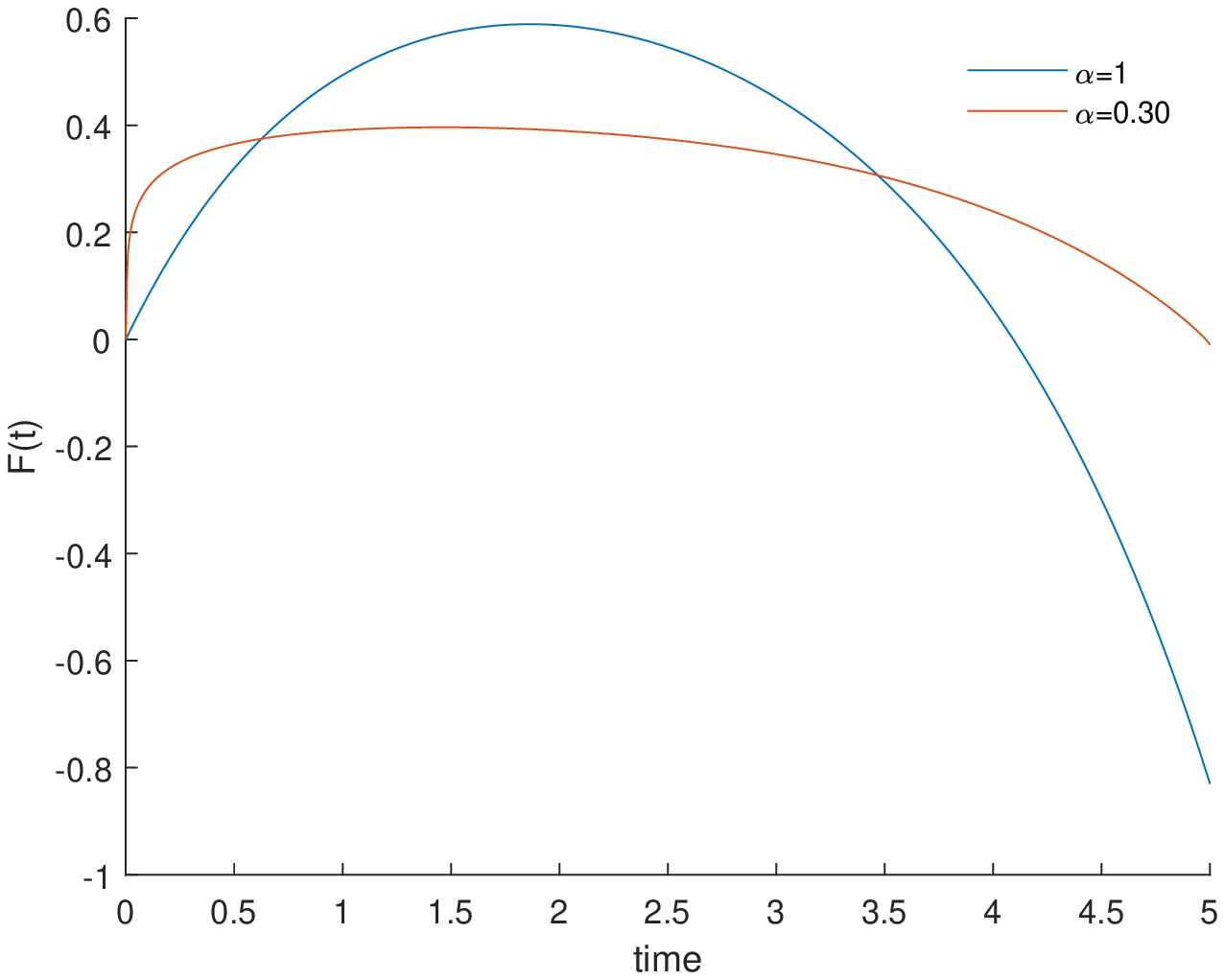}
\caption{Efficacy function $F(t)$ \eqref{efficacy_function}.}\label{fig:efficacy:023}
\end{subfigure}
\caption{Control function $v_1(t)$ and efficacy function $F(t)$ associated
to the FOCP  \eqref{14}--\eqref{optim:cond}, with values from Table~\ref{tab:param},
weights $B_1=B_2=40$, and the fractional order derivatives $\alpha=1.0$  and $0.30$.}
\label{fig:v1__F:023}
\end{figure}


\section{Conclusion}
\label{sec:conc}

We investigated the dynamics involved in HIV-AIDS infection 
by means of a fractional SICA model with Caputo's fractional derivatives.
Besides determining the points of stability of the system, we also applied 
fractional optimal control to virtualize determined scenarios and choose 
a strategy that minimizes the spreading of the disease.

The infection process was modelled by a general functional response.
This general incidence function $f(S,I)$ was used
to compute the basic reproduction number under biologically
reasonable hypothesis. Under such hypothesis, 
function $f$ covers many types of incidence functions 
existing in the literature, such as the bilinear incidence rate, 
saturated incidence rate, Beddington--DeAngelis, Crowley--Martin, 
and Hattaf--Yousfi functional responses. 
In numerical simulations, we have chosen 
the bilinear incidence rate $f(S,I)=\beta S$, 
which has shown a better fitting to real data of Morocco.
 
Stability and instability of equilibrium points were determined in terms 
of the basic reproduction number. Then, a fractional optimal control system 
was formulated and the best strategy for minimizing
the spread of the disease into the population
was determined through numerical simulations based
on the derived necessary optimality conditions.

Application of optimal control to the fractional SICA model shows 
that HIV treatment is effective on the reduction of infected individuals. 
Also, our results show that treating people with AIDS symptoms is useless 
when we can act over infected people with ART treatment. The modification of the value
of the fractional derivative order, $\alpha$, corresponds to variations of solutions 
to the fractional optimal control problem. When treatment becomes expensive, 
the fractional model ($\alpha<1$) is more appropriated due to its effectiveness
and because it provides control measures not as expensive as the ones given by
the classical model.


\section*{Acknowledgements}

Rosa was funded by The Portuguese Foundation for Science and Technology
(FCT -- Funda\c{c}\~ao para a Ci\^encia e a Tecnologia) through national
funds under the project UIDB/EEA/50008/2020;
Torres by FCT through CIDMA, reference UIDB/04106/2020.
The authors would like to acknowledge the comments and suggestions from three
anonymous reviewers, which helped them to enriched their work. 



\end{document}